\documentclass[10pt]{article}

\usepackage[letterpaper, left=1in, top=1in, right=1in, bottom=1in, verbose, ignoremp]{geometry}

\usepackage{url}\RequirePackage[colorlinks,citecolor=blue, linkcolor=blue,urlcolor = blue]{hyperref}
\usepackage{latexsym,amssymb,amsmath,amsfonts,graphicx,color,fancyvrb,amsthm,enumerate,subfigure,mathrsfs}
\usepackage[authoryear,round]{natbib}
\usepackage[dvipsnames]{xcolor}
\usepackage{xy}\xyoption{all} \xyoption{poly} \xyoption{knot}
\usepackage{float}
\usepackage{bm}
\usepackage{bbm}
\usepackage{multicol,multirow}
\usepackage{booktabs}
\usepackage{array}
\usepackage{relsize}
\usepackage{chngcntr}
\usepackage{etoolbox}
\usepackage{caption}
\usepackage{tikz}
\usetikzlibrary{patterns}
\usepackage{amsopn}
\usepackage{calc}
\usepackage{algorithm}
\usepackage[noend]{algpseudocode}

\usepackage{hyperref}

\newtheorem{theorem}{Theorem}[]
\newtheorem*{theorem*}{Theorem}

\newtheorem{lemma}[theorem]{Lemma}
\newtheorem{proposition}[theorem]{Proposition}

\newtheorem*{claim*}{Claim}
\theoremstyle{definition}
\newtheorem{definition}[theorem]{Definition}
\newtheorem*{definition*}{Definition}
\theoremstyle{AppDefinition}

\theoremstyle{AppClaim}

\theoremstyle{remark}

\newtheorem{example}[theorem]{Example}
\newtheorem*{example*}{Example}

\def\beginmat{ \left( \begin{array} }
\def\endmat{ \end{array} \right) }

\def\log{{\rm log}}

\newcommand{\dtr}{{d_{\mathrm{tr}}}}

\makeatletter
\newcommand*{\op}{%
  \DOTSB
  \mathop{\vphantom{\bigoplus}\mathpalette\matt@op\relax}%
  \slimits@
}
\newcommand\matt@op[2]{%
  \vcenter{\m@th\hbox{\resizebox{\widthof{$#1\bigoplus$}}{!}{$\boxplus$}}}%
}
\makeatother

\newcommand{\one}{{\bf{1}}}

\def\R{{\mathbb R}}

\makeatletter
\def\@biblabel#1{}
\makeatother

\makeatletter
\patchcmd{\NAT@citex}
  {\@citea\NAT@hyper@{%
     \NAT@nmfmt{\NAT@nm}%
     \hyper@natlinkbreak{\NAT@aysep\NAT@spacechar}{\@citeb\@extra@b@citeb}%
     \NAT@date}}
  {\@citea\NAT@nmfmt{\NAT@nm}%
   \NAT@aysep\NAT@spacechar\NAT@hyper@{\NAT@date}}{}{}

\patchcmd{\NAT@citex}
  {\@citea\NAT@hyper@{%
     \NAT@nmfmt{\NAT@nm}%
     \hyper@natlinkbreak{\NAT@spacechar\NAT@@open\if*#1*\else#1\NAT@spacechar\fi}%
       {\@citeb\@extra@b@citeb}%
     \NAT@date}}
  {\@citea\NAT@nmfmt{\NAT@nm}%
   \NAT@spacechar\NAT@@open\if*#1*\else#1\NAT@spacechar\fi\NAT@hyper@{\NAT@date}}
  {}{}

\makeatother

\begin{document}
\def\spacingset#1{\renewcommand{\baselinestretch}%
{#1}\small\normalsize} \spacingset{1}
\begin{flushleft}
{\Large{\textbf{Tropical Geometry of Phylogenetic Tree Space:\\ A Statistical Perspective}}}
\newline
\\
Anthea Monod$^{1,\dagger}$, Bo Lin$^{2}$, Ruriko Yoshida$^{3}$, and Qiwen Kang$^{4}$
\\
\bigskip
\bf{1} Department of Mathematics, Imperial College London, UK 
\\
\bf{2} School of Mathematics, Georgia Institute of Technology, Atlanta, GA, USA
\\
\bf{3} Department of Operations Research, Naval Postgraduate School, Monterey, CA, USA
\\
\bf{4} Medpace, Inc.~and Department of Statistics, University of Kentucky, Lexington, KY, USA
\\
\bigskip
$\dagger$ Corresponding e-mail: a.monod@imperial.ac.uk\\
\end{flushleft}


\section*{Abstract}

Phylogenetic trees are the fundamental mathematical representation of evolutionary processes in biology.  They are also objects of interest in pure mathematics, such as algebraic geometry and combinatorics, due to their discrete geometry.  Although they are important data structures, they face the significant challenge that sets of trees form a non-Euclidean phylogenetic tree space, which means that standard computational and statistical methods cannot be directly applied.  In this work, we explore the statistical feasibility of a pure mathematical representation of the set of all phylogenetic trees based on tropical geometry for both descriptive and inferential statistics, and unsupervised and supervised machine learning.  Our exploration is both theoretical and practical.  We show that the tropical geometric phylogenetic tree space endowed with a generalized Hilbert projective metric exhibits analytic, geometric, and topological properties that are desirable for theoretical studies in probability and statistics and allow for well-defined questions to be posed.  We illustrate the statistical feasibility of the tropical geometric perspective for phylogenetic trees with an example of both a descriptive and inferential statistical task.  Moreover, this approach exhibits increased computational efficiency and statistical performance over the current state-of-the-art, which we illustrate with a real data example on seasonal influenza.  Our results demonstrate the viability of the tropical geometric setting for parametric statistical and probabilistic studies of sets of phylogenetic trees.

\paragraph{Keywords:} BHV tree space; phylogenetic tree space; tree metric; tropical geometry; tropical metric.


\section{Introduction}
\label{sec:intro}

Evolutionary relationships describing how organisms are related by a common ancestor are represented in a branching diagram known as a {\em phylogenetic tree}.  Phylogenetic trees model many important and diverse biological phenomena, such as speciation, the spread of pathogens, and cancer evolution.  Methodology to analyze phylogenetic datasets has been under active research for several decades for two important reasons.  First, explicit computations directly on collections of phylogenetic trees are challenging due to high dimensionality in terms of a large number of leaves, a long evolutionary history, and an intricate branching pattern.  Second, standard statistical methodologies are not directly applicable due to the non-Euclidean nature of the trees themselves as well as the set that they make up.  Significant previous work addresses various classical statistical interests, however a fundamental breakthrough for quantitative studies on sets of trees emerged through studying the geometry of the set of all phylogenetic trees \citep{BILLERA2001733}.  

Referred to in the literature as {\em BHV tree space}---after the authors Billera, Holmes, and Vogtmann---the set of all phylogenetic trees is studied in a setting where each tree is represented as an individual point.  The geometry is characterized by a unique geodesic between any two points; its length defines a metric on the space.  Since its introduction in 2001, it has been actively studied in various wide-reaching domains, including algebraic geometry \citep[e.g.,][]{DEVADOSS201575}, category theory \citep[e.g.,][]{baez2017operads}, computational biology \citep[e.g.,][]{Weyenberg2016}, and statistical genetics \citep[e.g.,][]{doi:10.1093/biomet/asx047}.  Despite its indisputable significance, the BHV geometry nevertheless poses significant data-analytic complications for both descriptive and inferential statistics and unsupervised and supervised machine learning.  Although the non-Euclidean aspect of tree space cannot be avoided, considering an alternative approach from pure mathematics based on {\em tropical geometry}---a variant of algebraic geometry---alleviates some of these complications and is a promising alternative approach for probability-based statistics on sets of phylogenetic trees.  The first formal connection between tropical geometry and mathematical phylogenetics arises in the space of phylogenetic trees in relation to a particular tropical algebraic variety \citep{Speyer2004}.  This coincidence has been further studied in theoretical research \citep[e.g.,][]{ARDILA200638, manon2011dissimilarity}, however, its implication and potential in applied work remain largely understudied and untapped.

In this paper, we explore the tropical geometric perspective of phylogenetic tree space with the aim of enabling descriptive and inferential statistics as well as unsupervised and supervised machine learning.  Specifically, we study the subspace of the {\em tropical projective torus} corresponding to the space of phylogenetic trees equipped with a generalized projective Hilbert metric, which we refer to as the {\em tropical metric}.  We refer to this metric space as {\em palm tree space} (tropical tree space) and show that it satisfies fundamental assumptions to ensure that probabilistic and parametric statistical questions are valid and well-defined, which establishes the grounds for future development of parametric studies in the tropical geometric setting for phylogenetic trees.  We give a concrete example of both a descriptive and inferential task in statistics: principal component analysis (PCA) and linear discriminant analysis (LDA), respectively.  We study these tasks in the context of both simulated and real-world seasonal influenza data.  Our real data application demonstrates that the tropical geometric approach exhibits improvements in computational efficiency and improved statistical performance over the BHV setting.


The remainder of this paper is organized as follows.  In Section \ref{sec:math}, we provide background and motivation for our study.  In Section \ref{sec:palm}, we discuss properties of the tropical metric on the space of phylogenetic trees and formally define palm tree space.  We study its geometry, topology, and analytic properties in relation to BHV space.  We also give some examples of theoretical probability measures used in statistics and that are important in probability theory.  In Section \ref{sec:app}, we give examples of statistical techniques on palm tree space, including numerical experiments and an application to real data in both palm tree and BHV space.  We close in Section \ref{sec:end} with a discussion, and some directions for future research.


\section{Background and Motivation}
\label{sec:math}

Phylogenetic trees are symbolic objects that model evolutionary divergence from a common ancestor.  In computational biology, the reconstruction of a phylogenetic tree from an input of sequence alignment data (e.g., DNA and RNA) is a challenging problem; reconstruction methods are known to be highly sensitive to the input sequences (different genes or coding regions will give rise to different trees), measurement errors (alignment or sequencing errors), and noise typical to this type of biological data \citep[e.g.,][]{10.1093/sysbio/syq073}.  This sensitivity naturally invites the question of how to compare trees, for example, arising from different reconstruction methods.  Mathematically, comparing objects entails measuring the distance between them; in the context of phylogenetic trees, this gives rise to a {\em tree space} equipped with a {\em metric between trees}.  One of the most significant challenges in computational work with phylogenetic trees as data objects is that their graphical structure gives rise to a non-Euclidean tree space.  

In this section, we provide the mathematical background to phylogenetic trees from the pure mathematical perspective and the statistical motivation for studying this perspective.

\subsection{Defining Phylogenetic Trees}

In what follows, we consider $N \in \mathbb{N}$.  A tree is an acyclic connected graph $T = (V,E)$, defined by a set of vertices $V$ and a set of edges $E$.  An {\em $N$-tree} is a tree with $N$ labeled terminal nodes called {\em leaves}.  Edges connecting to leaves are called {\em external edges}, other edges are called {\em internal edges}.  A {\em binary $N$-tree} is an $N$-tree with the following conditions on the degree of a vertex $v \in V$: if $\operatorname{deg}(v) = 2$, then $v$ is the {\em root} of the tree and it is unique; if $\operatorname{deg}(v) = 1$, then $v$ is a leaf; and if $\operatorname{deg}(v) = 3$, then $v$ is an internal vertex.  A {\em tree topology} is the ``shape" of the tree; it is a branching configuration of edges together with a leaf labelling scheme.  There are $(2N-3)!!$ binary tree topologies on $N$ leaves \citep{schroder1870vier}.  A {\em metric $N$-tree} is a tree with zero or positive lengths on all edges; metric $N$-trees are also known as {\em phylogenetic trees}.  We denote the space of phylogenetic trees with $N$ leaves by $\mathcal{T}_N$.

A phylogenetic tree may be represented by all pairwise distances between leaves.  Let $w_{ij}$ denote the distance between leaves $i$ and $j$, given by the sums of edge lengths along the unique path between $i$ and $j$.  The $N \times N$ matrix $W$ with entries $w_{ij}$ then represents a phylogenetic tree.  Since $W$ is symmetric with zeros along the diagonal, the upper-triangular portion of the matrix contains all of the unique information needed to specify a phylogenetic tree in terms of its pairwise distances.  Thus $W$ represents a phylogenetic tree as, in essence, a distance matrix.  
However, in order for the distance matrix $W$ to represent a phylogenetic tree, the following additional condition must be satisfied.

\begin{proposition}{(Four-Point Condition \cite[Theorem 1]{BUNEMAN197448}, \cite[Lemma 4.3.6]{maclagan2015introduction}).}
\label{prop:4pt}
A distance matrix $W$ represents a phylogenetic tree if it satisfies the conditions to be a metric and the maximum among the following {\em Pl\"{u}cker relations} is attained at least twice for $1 \leq i < j < k < \ell \leq N$:
\begin{equation}
\label{eq:plucker}
w_{ij} + w_{k\ell}, \qquad w_{ik} + w_{j\ell}, \qquad w_{i\ell} + w_{jk},
\end{equation}
or equivalently, that 
\begin{equation}
\label{eq:4pt}
w_{ij} + w_{k\ell} \leq \max(w_{ik} + w_{j\ell},\, w_{i\ell} + w_{jk})
\end{equation}
for all $i,j,k,\ell \in \{1, 2, \ldots, N\}$.
\end{proposition}

A distance matrix $W$ satisfying the conditions of Proposition \ref{prop:4pt} is known as a {\em tree metric}.  Note that tree metrics represent phylogenetic trees; these differ from metrics between trees.  For $W$ representing a phylogenetic tree (that is, satisfying the conditions in Proposition \ref{prop:4pt}), we may also represent the upper triangular portion of the matrix in vector form by setting $n := \binom{N}{2}$ and defining the following map:
\begin{equation}
\label{eq:W}
\begin{aligned}
\mathcal{W}: \mathcal{T}_N & \rightarrow \mathbb{R}^n,\\
W &\mapsto w = (w_{12},\, w_{13},\, \ldots,\, w_{1N},\, w_{23},\, \ldots,\, w_{2N},\, \ldots,\, w_{(N-1)N}),
\end{aligned}
\end{equation}
where the entries are ordered lexicographically.

\begin{example}\label{ex:4pt}
The tree metric $w \in \mathbb{R}^6$ for the tree in Figure \ref{fig:4ptex} expressed as a vector is $(w_{PQ},\, w_{PR},\, w_{PS},\, w_{QR},\, w_{QS},\, w_{RS})$.  As a matrix $W$, it is
$$
\begin{pmatrix} 0 & w_{PQ} & w_{PR} & w_{PS}\\ & 0 & w_{QR} & w_{QS} \\  & & 0 & w_{RS}\\ & & & 0 \end{pmatrix} 
= \begin{pmatrix} 0 & a + b & a + c + d & a + c + e\\ & 0 & b + c + d & b + c + e \\  & & 0 & d + e\\ & & & 0 \end{pmatrix}.
$$

The Pl\"{u}cker relations (\ref{eq:plucker}) associated with $W$ are
\begin{equation*}
\begin{split}
A:= w_{PQ} + w_{RS} & = a + b + d + e,\\
B:= w_{QR} + w_{PS} & = a + b + 2c + d + e,\\
C:= w_{PR} + w_{QS} & = a + b + 2c + d + e.
\end{split}
\end{equation*}
Since $P < Q < R < S$, the maximum $B = C$ is achieved exactly twice, and $B-A = 2c > 0$.  Also, (\ref{eq:4pt}) holds: $A \leq \max\{ B, C \} = B.$

\begin{figure}[ht]
\centering
\begin{tikzpicture}[scale=0.25]
\draw (0,0) -- (4,3) -- (0,4.5);
\filldraw [black] (0,0) circle (1pt);
\filldraw [black] (0,4.5) circle (1pt);
\node [below] at (0,0) {$P$};
\node [below] at (2,1.4) {$a$};
\node [above] at (0,4.5) {$Q$};
\node [above] at (2,3.9) {$b$};
				\draw (4,3) -- (9,3);
				\node [below] at (6.5,3) {$c$};
				\draw  (11,0) -- (9,3) -- (13,5);
				\filldraw [black] (11,0) circle (1pt);
				\filldraw [black] (13,5) circle (1pt);
				\node [above] at (13,5) {$R$};
				\node [above] at (10.9,4) {$d$};
				\node [below] at (11,0) {$S$}; 
				\node [below] at (9.9,1.4) {$e$};
			\end{tikzpicture}
			\caption{Example of an unrooted phylogenetic tree to illustrate the four-point condition.}
			\label{fig:4ptex}
	\end{figure}
\end{example}

The four-point condition for trees may be further tightened to give rise to an important subclass of trees as follows.

\begin{proposition}[Three-Point Condition \citep{JARDINE1967173,maclagan2015introduction}]
\label{prop:3pt}
A distance matrix $W$ represents an {\em ultrametric tree} if it satisfies the conditions to be a tree metric and the maximum among $w(i, j),\,  w(i, k),\, w(j,k)$
is achieved at least twice for $1 \leq i < j < k \leq N$, or equivalently, that
$
w(i,k) \leq \max(w(i,j),\, w(j,k))
$ 
for all $i,j,k\in [N]$.
\end{proposition}

Denote the space of all phylogenetic trees satisfying the three-point condition with $N$ leaves by $\mathcal{U}_N$.  Ultrametric trees are also {\em equidistant trees}; i.e., rooted trees where the distance from the root to every leaf is equal.

\begin{proposition}\label{prop:ultra}
A phylogenetic tree $T$ is an ultrametric tree if and only if $T$ is equidistant.
\end{proposition}

\begin{proof}
For any two points $X,Y$ on a tree, we denote by $w(X,Y)$ the length of the unique path connecting $X$ and $Y$.  Suppose a tree $T$ is equidistant with root $R$.  Then for any three leaves $A,B,C$, we have that $w(R,A)=w(R,B)=w(R,C)$.  Since $R,A,B,C$ satisfies the four-point condition, the maximum among (\ref{eq:plucker}), $w(R,A)+w(B,C), w(R,B)+w(A,C),$ and $w(R,C)+w(A,B)$ is attained at least twice.  Thus, the maximum among $w(B,C), w(A,C), w(A,B)$ is also attained at least twice, satisfying the three-point condition, and $T$ is therefore an ultrametric.



	
Conversely, suppose $T$ is an ultrametric.  Then there are finitely many leaves in $T$, so we can choose a pair of leaves $A,B$ such that $w(A,B)$ is maximal among all such pairs.  Along the unique path from $A$ to $B$, there is a unique point $R$ such that $w(R,A)=w(R,B)$.  For any other leaf $C$, consider the distance $w(R,C)$: Since the paths from $R$ to $A$ and $B$ only intersect at $R$, the path from $R$ to $C$ intersects at least one of them only at $R$.  Suppose without loss of generality that the path from $R$ to $A$ is such a path.  Then $w(A,C)=w(R,A)+w(R,C)$.  Since $w(A,C)\le w(A,B)$, we have $w(R,C)\le w(R,B)$.  If $w(R,C)<w(R,B)$, then $w(A,C)<w(R,A)+w(R,B)=2w(R,B)$, and $w(B,C)\le w(R,B)+w(R,C)<2w(R,B)$, so the maximum among $w(A,B),\, w(A,C),\, w(B,C)$ is $w(A,C)=2w(R,B)$ and it is only attained once --- a contradiction, since $T$ was assumed an ultrametric.  Hence $w(R,C)=w(R,B)$, and $R$ has equal distance to all leaves of $T$.  Therefore $T$ is equidistant with root $R$.
\end{proof}

\subsection{Metrics on Tree Spaces: BHV Space}

Various metrics between trees have been derived in biology.  A notable class of metrics strives to retain the inner product property akin to Euclidean distance, which makes them popular due to their integrability into a wide range of statistical approaches, such as functional and nonparametric modeling.  One metric from this class extensively used in biology is the {\em Robinson--Foulds metric} \citep{ROBINSON1981131}.  This metric (and other inner-product distances between trees) is known to suffer from structural and interpretive errors.  For example, many pairs of trees measure the same distance apart; also, large distances between trees, counterintuitively, do not necessarily indicate a disparity in ancestral heritage \citep{Steel1993}.  Other commonly-occurring distances include the nearest neighbor interchange metric \citep{waterman1976some}, subtree transfer distance \citep{Allen2001}, and variational distance \citep{steel2006variational}.  For a detailed review of metrics between trees, see \cite{Weyenberg2016, 10.1093/sysbio/syw025}.  A pioneering approach that bypasses difficulties and limitations of these metrics focuses on the geometry of tree space \citep{BILLERA2001733}.  

\begin{figure}[ht]
\centering
\includegraphics[scale=0.4]{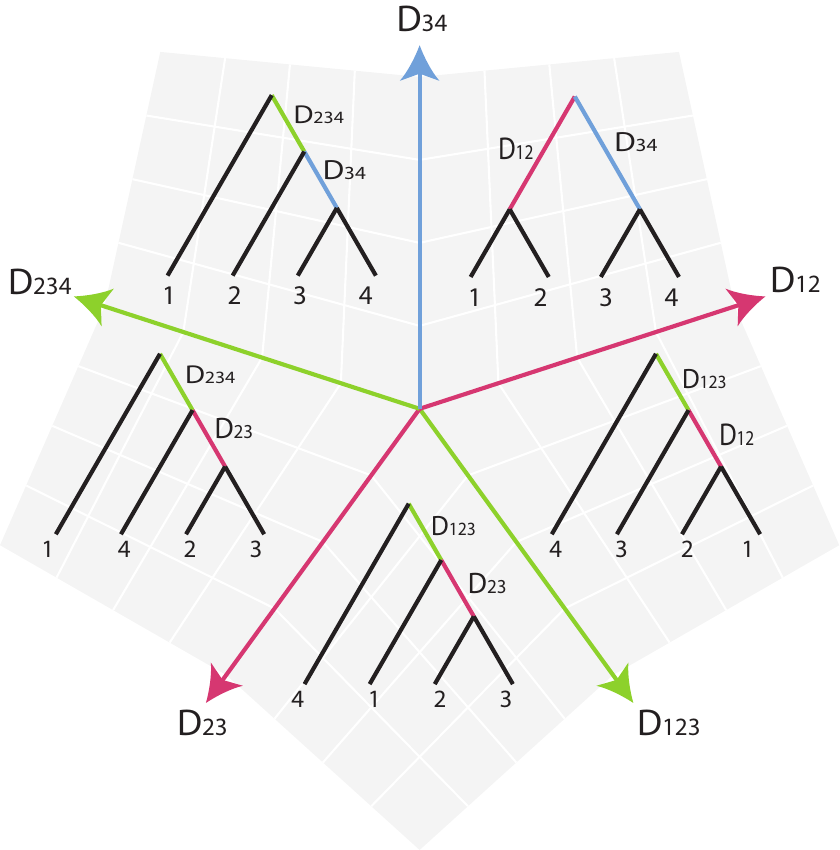}
\caption{Some orthants in $\mathcal{T}_4^{\operatorname{BHV}}$.  The axes $D_{ijk}$ represent the coordinate corresponding to the length of the internal edge leading to the $\{ijk\}$ clade.}
			\label{fig:BHV4}
\end{figure}

Specifically, the space of phylogenetic trees is modeled as a {\em moduli space}, where each point in the space represents a phylogenetic tree.  Trees are expressed only by the lengths of their internal edges, which are recorded as entries in a vector of dimension $N-2$ since in a binary tree with $N$ leaves, there are at most $N-2$ internal edges; see Figure \ref{fig:BHV4} for an illustration.  External edges are not considered, since taking them into account does not affect the geometry of the space: including external edges simply amounts to taking the product of tree space with an $N$-dimensional Euclidean space.  A nonnegative Euclidean orthant $\mathbb{R}^{N-2}_{\geq 0}$ is associated to each tree topology.  BHV space may also be interpreted combinatorially: For each orthant, the {\em link of the origin}
\begin{equation}
\label{eq:link}
\mathcal{L}_N := \big\{ (x_1,\, \ldots,\, x_{N-2}) \mid \sum_{i} x_i = 1\big\}
\end{equation}
gives rise to a simplicial complex of dimension $N-3$ \citep{BILLERA2001733}.  BHV space is an infinite cone over $\mathcal{L}_N$.  

The $(2N-3)!!$ orthants are grafted at right-angles to make up the tree space, which gives rise to a property of nonpositive curvature known as $\operatorname{CAT}(0)$.  In $\operatorname{CAT}(0)$ spaces, there is a unique shortest path between any two points; here, this is the {\em BHV geodesic}.  To compute BHV geodesics, first, the geodesic distance between two trees (represented by only internal edge lengths as in the original reference \citep{BILLERA2001733}) is computed, and then the external branch lengths are factored in to compute the overall geodesic distance between two trees in full generality (where external branch lengths are included, as opposed to the original reference \citep{BILLERA2001733}), by taking the differences between external branch lengths.  Since each orthant is locally viewed as a Euclidean space, the shortest path between two points within a single orthant is a straight Euclidean line.  The difficulty appears in establishing which sequence of orthants joining the two topologies contains the geodesic.  In the case of four leaves, this can be readily determined using a systematic grid search, but such a search is intractable with larger trees.  Currently, the fastest available algorithm to compute geodesic paths between any two points in this tree space has a time complexity of $O(N^4)$ \citep{Owen:2011:FAC:1916480.1916603}.  The length structure of the BHV geodesic induces the {\em BHV metric} $d_{\mathrm{BHV}}$ on this space.  This setup has come to be known as {\em BHV space} $\mathcal{T}_N^{\operatorname{BHV}}$ and is ubiquitous even in non-biological fields, including computer vision, combinatorics, and category theory.  It has also been proposed as the definitive setting for computational studies on sets of phylogenetic trees \citep{GAVRYUSHKIN2016197}.

It turns out that BHV space poses considerable limitations for classical descriptive and inferential statistics.  On the descriptive front, the convex hull of finitely many points in tree space with edges given by BHV geodesics is unbounded in dimension \citep{doi:10.1137/16M1079841}, so there exists no obvious subspace for projections and no lower dimensional representations of data.  This is restrictive for classical dimensionality reduction and data visualization methods, such as principal component analysis (PCA).  An important challenge in inferential statistics in BHV space concerns the Fr\'{e}chet mean, which is a fundamental quantity in statistics: the Fr\'{e}chet mean is the generalization of the classical arithmetic mean to arbitrary metric spaces; see also Section \ref{subsec:means} further on.  In BHV space, Fr\'{e}chet means are {\em sticky}: the mean fails to be injective and ``sticks'' to lower dimensional strata \citep{10.2307/42919709}; see Example \ref{ex:sticky}.  Thus, perturbing points in a sample results in no change in the mean, meaning that exact parametric asymptotic results cannot be derived, which prohibits classical exact statistical inference.  

\begin{example}
\label{ex:sticky}
In Figure \ref{fig:sticky}, we position three unit masses on the $3$-spider as shown below, which is the stratified space of three $\mathbb{R}_{\geq 0}$ rays joined at the origin.  Two of the masses are at distance 1 from the origin, while the remaining mass is at distance $a$ from the origin.  This is precisely the BHV space of phylogenetic trees with three leaves and fixed external edge lengths.  The position $x$ of the barycenter (Fr\'echet mean) is calculated by minimizing $2(1+x)^2 + (a-x)^2$.  The solution is $x = 0$ for $a < 2$, and $x = (a-2)/3$ for $a \geq 2$.  The Fr\'{e}chet mean tends to stick to lower-dimensional strata. 
	\begin{figure}[h]
			\centering
			\begin{tikzpicture}[scale=0.4]
				\draw (0,0) -- (3,2) -- (3,6);
				\draw [<->] (1, 0.85) -- (2.9, 2.1);
				\draw [<->] (3.15, 2.1) -- (3.15, 4.95);
				\draw (3,2) -- (6,0);
				\draw [<->] (3.1, 2.1) -- (5, 0.85);
				\draw [fill] (1,0.67) circle [radius=0.05];
				\node [above] at (1.75,1.45) {$1$};
				\draw [fill] (5,0.67) circle [radius=0.05];
				\node [above] at (4.4,1.4) {$1$};
				\draw [fill] (3,5) circle [radius=0.05];
				\node [right] at (3.1,4) {$a$}; 
				\draw [fill] (3,3.5) circle [radius=0.05];
				\node [left] at (3,3.5) {$x$};
			\end{tikzpicture}
			\caption{$3$-spider to illustrate stickiness.}
			\label{fig:sticky}
	\end{figure}
\end{example}

Sophisticated methods have been developed to bypass these difficulties: a locus of BHV Fr\'{e}chet means has well-behaved dimensionality and serves as a suitable lower dimensional projective space \citep{doi:10.1093/biomet/asx047}, and a central limit theorem for BHV Fr\'{e}chet means exists via a generalized delta method \citep{Barden2018}.  Inferential techniques have also been proposed based on this generalized delta method strategy \citep[e.g.,][]{willis2019confidence, willis2018uncertainty}.  These, and other proposed methods, are largely approximate, rather than exact statistical methods; additionally, they tend to be nonparametric, rather than parametric.  These statistical challenges have spurred recent proposals of alternative tree spaces \citep{garba2021information}.

\subsection{Tropical Geometry and Phylogenetic Tree Space}

In this work, we focus on the appearance of phylogenetic tree space in tropical geometry in groundbreaking work that formally connects the space of phylogenetic trees and the {\em tropical Grassmannian} \citep{Speyer2004}.  We now outline the connection between tropical geometry and phylogenetic tree space.  To do this, we return to the map $\mathcal{W}$ (\ref{eq:W}).  Specifically, we would like to understand what the image of $\mathcal{W}$ is: if it is a linear space, then theory from linear algebra is applicable; if it is a manifold, then principles of Riemannian geometry may be applied.  It turns out that the image of $\mathcal{W}$ is tropical geometric, so new tools for statistics are needed.

To see this, notice that the embedding (\ref{eq:W}) of trees into Euclidean space may be refined: if we do not wish to distinguish between phylogenetic trees differing by a constant on each external edge, we may consider the quotient space $\mathbb{R}^n/\mathbb{R}\one$, where $\one$ is the all-one vector $(1,1,\ldots,1)$, which gives a reduction in dimension.  The quotient space $\mathbb{R}^n/\mathbb{R}\one$ is known as the {\em tropical projective torus} and it is generated by an equivalence relation $\sim$ specifying that for two points $x,y \in \mathbb{R}^n$, $x \sim y$ if and only if all coordinates of their difference $x-y$ are equal.  In the context of trees, the quotient normalizes evolutionary time between trees.  The tropical projective torus is the ambient space of the space of phylogenetic trees; $\mathcal{T}_N$ is a proper subset of $\mathbb{R}^n/\mathbb{R}\one$.  The tropical projective torus $\mathbb{R}^n/\mathbb{R}\one$ may also be generated by a group action: Let $G := \{ (c, c, \ldots, c) \in \mathbb{R}^n \mid c \in \mathbb{R} \}$ with coordinate-wise addition, then $G$ is an additive group.  $G$ acts on $\mathbb{R}^n$ as follows: for $g \in G$ and $x \in \mathbb{R}$,
$$
g \circ x = (x_1 + g_1,\, x_2 + g_2,\, \ldots,\, x_n + g_n).
$$
Each point in $\mathbb{R}^n/\mathbb{R}\one$ is then exactly one orbit under the group action of $G$ on $\mathbb{R}^n$.

Furthermore, if we disregard differences on external edges, we may consider the quotient space $\mathbb{R}^n/\operatorname{im}(\varphi)$ where the map $\varphi:\mathbb{R}^N \rightarrow \mathbb{R}^n$ is given by $\varphi(x_1,\ldots,x_N) = (x_1 + x_2,\, x_1 + x_3,\, \ldots,\, x_{N-1} + x_N)$.  This map has the geometric intuition that two trees are identified if the lengths of the $N$ edges adjacent to the leaves are the only difference between them, which gives intuition on why this map is well-defined; for further technical details on why this map is well-defined, please see page 396 of \cite{Speyer2004}.  We thus obtain the following sequence of maps:
\begin{equation}
\label{eq:maps}
\mathcal{T}_{N-1} \rightarrow \mathbb{R}^n \rightarrow \mathbb{R}^n/\mathbb{R}\one \rightarrow \mathbb{R}^n/\operatorname{im}(\varphi).
\end{equation}

In algebraic geometry, the solution sets of systems of polynomial equations---referred to as {\em algebraic varieties}---are studied.  In {\em tropical geometry}, these polynomial equations are defined in the {\em tropical semiring}, $(\mathbb{R}\cup\{\infty\}, \oplus, \odot)$ where $a \oplus b := \min(a,b)$ and $a \odot b := a+b$.  Tropical mathematics involves studying various mathematical objects and problems which are defined using these operations.  For example, let $a^N$ denote the tropical product of $a$ with itself $N$ times; let $\mathcal{A} \subset \mathbb{N}^N$.  Tropical polynomials are piecewise linear functions:
$$
f(x_1, \ldots, x_N) = \bigoplus_{a \in \mathcal{A}} c_a \odot x_1^{a_1} \odot \cdots \odot x_N^{a_N} = \min_{a\in\mathcal{A}} (c_a + a_1x_1 + \cdots + a_N x_N).
$$
A {\em tropical hypersurface} $\mathcal{H}(f)$ is the set of all $(x_1,\ldots, x_N) \in \mathbb{R}^N$ where $f$ is attained at least twice as $a$ runs over $\mathcal{A}$.  

Notice that the Pl\"{u}cker relations (\ref{eq:plucker}) given in Proopsition \ref{prop:4pt} are tropical polynomials, and thus, the set of all phylogenetic trees constitutes a tropical hypersurface with min replaced by max.  Note also that the max-plus semiring $(\mathbb{R} \cup \{-\infty\}, \boxplus, \odot)$, where $a \boxplus b := \max(a,b)$, is isomorphic to the tropical semiring.  Thus, the four-point condition defining phylogenetic trees is tropical.

In algebraic geometry, the real Grassmannian $G_{2,N}$ is the following projective variety in the projective space $\mathbb{P}^{N-1}$:
$$
G_{2,N} = \big\{(x_{12},\, x_{13},\, \ldots,\, x_{(N-1)N}) \in \mathbb{P}^{n-1} \mid x_{ij}x_{k\ell} - x_{ik}x_{j\ell} + x_{i\ell}x_{jk} = 0 \mathrm{~for~} 1\leq i < j < k < \ell \leq N\big\}.
$$
The {\em tropical Grassmannian} $\mathcal{G}_{2,N}$ is then obtained by replacing the polynomial by its equivalent with standard operations replaced by their tropical equivalents (a process often referred to as ``tropicalization'') and the vanishing set by intersections of tropical hypersurfaces.  In other words, $\mathcal{G}_{2,N}$ is given by the intersection of tropical hypersurfaces $\mathcal{H}(x_{ij}\odot x_{k\ell} \oplus x_{ik}\odot x_{j\ell} \oplus x_{i\ell} \odot x_{jk})$ for $1 \leq i < j < k < \ell \leq N$.

To visualize $\mathcal{G}_{2,N}$, we have the following behavior of images through the sequence of maps (\ref{eq:maps}): the image of $\mathcal{G}_{2,N}$ in $\mathbb{R}^n/\mathbb{R}\one$ is a fan $\mathcal{G}'_{2,N}$ of dimension $(2N-2)$; the image of $\mathcal{G}'_{2,N}$ in $\mathbb{R}^n/\operatorname{im}(\varphi)$ is a fan $\mathcal{G}''_{2,N}$ of dimension $N-3$; and intersecting $\mathcal{G}''_{2,N}$ with the unit sphere yields a polyhedral complex $\mathcal{G}'''_{2,N}$, where each facet $\mathcal{G}'''_{2,N}$ is a polytope of dimension $N-4$.  It turns out that $\mathcal{G}''_{2,N}$ coincides with $\mathcal{T}_{N-1}^{\mathrm{BHV}}$, $\mathcal{G}'''_{2,N}$ coincides with $\mathcal{L}_{N-1}$ (\ref{eq:link}), and the image of $\mathcal{W}$ is precisely the tropical Grassmannian $\mathcal{G}_{2,N}$ \citep{Speyer2004}.  

\begin{example}
As an illustrative example, we study the case of $N = 4$ leaves: $\mathcal{G}_{2,4}$ is the hypersurface $\mathcal{H}(x_{12} \odot x_{34} \oplus x_{13} \odot x_{24} \oplus x_{14} \odot x_{23})$, which is the collection of points such that at least one of the following systems holds: $x_{12} + x_{34} = x_{13} + x_{24} \leq x_{14} + x_{23},\,\, x_{12} + x_{34} = x_{14} + x_{23} \leq x_{13} + x_{24},\,\, x_{14} + x_{23} = x_{13} + x_{24} \leq x_{12} + x_{34}.$
\begin{figure}
\centering
\includegraphics[scale=0.25]{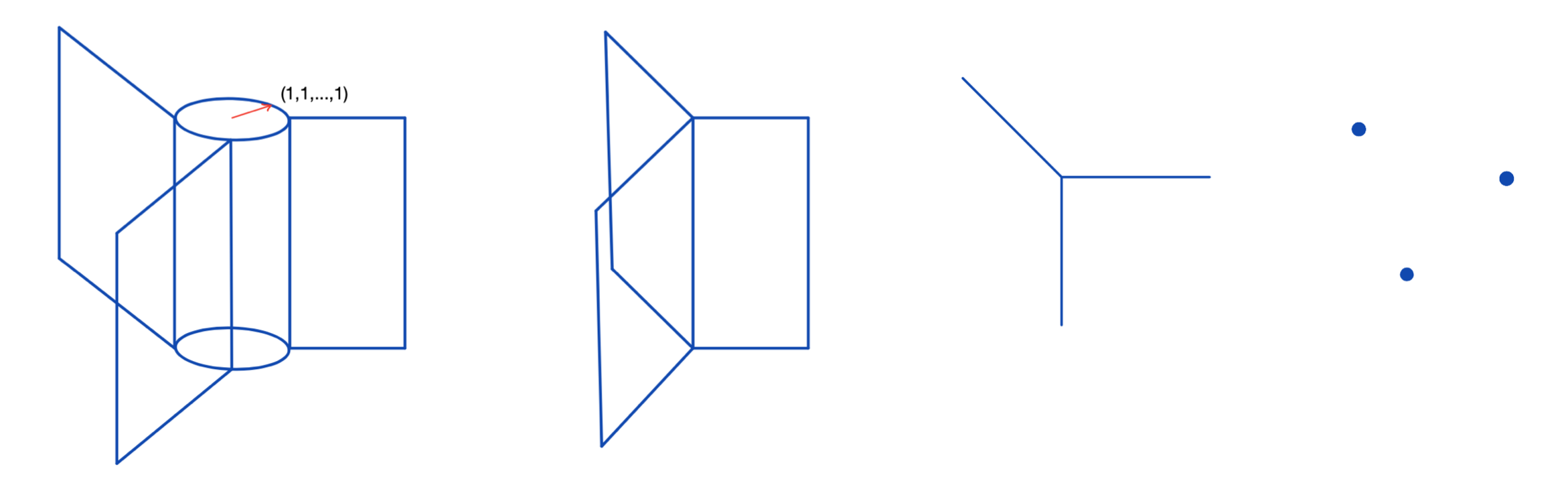}
\caption{Visualizing the tropical Grassmannian $\mathcal{G}_{2,4}$.  From left to right, we have the images of $\mathcal{G}_{2,4}$, $\mathcal{G}'_{2,4}$, $\mathcal{G}''_{2,4}$, and $\mathcal{G}'''_{2,4}$ under \eqref{eq:maps}.  Notice that $\mathcal{G}''_{2,4}$ is $\mathcal{T}_3^{\mathrm{BHV}}$ and $\mathcal{G}'''_{2,4}$ is $\mathcal{L}_3$.}
\label{fig:tropgrass4} 
\end{figure}
For each system, equality determines a 5-dimensional hyperplane in $\mathbb{R}^6$, while inequality determines a closed half-space in $\mathbb{R}^6$.  Their intersection is isomorphic to $\mathbb{R}^4 \times \mathbb{R}_{\geq 0}$.  Since there are three systems, $\mathcal{G}_{2,4}$ is the union of three copies of $\mathbb{R}^4 \times \mathbb{R}_{\geq 0}$ glued along the space $x_{12} + x_{34} = x_{13} + x_{24} = x_{14} + x_{23}$, which is the image of $\varphi: \mathbb{R}^4 \rightarrow \mathbb{R}^6$.  $\mathcal{G}''_{2,4}$ then consists of three copies of $\mathbb{R}_{\geq 0}$ (i.e., $\mathcal{T}_3^{\mathrm{BHV}}$; see also Example \ref{ex:sticky}) and $\mathcal{G}'''_{2,4}$ consists of three points (i.e., $\mathcal{L}_3$).
\end{example}


\section{Palm Tree Space}
\label{sec:palm}

A fundamental requirement to comparative and statistical studies on the tropical geometric interpretation of phylogenetic tree space is a metric.  On the tropical projective torus, a {\em generalized Hilbert projective metric} has been used in other settings \citep[e.g.,][]{joswig2007josephine, AKIAN20113261, COHEN2004395}.  We adapt this metric in our studies and refer to it as the {\em tropical metric}.  

In this section, we review the tropical metric and study its properties, especially in relation to the BHV metric.  We then present our main contribution, which is a formal and theoretical study of mathematical properties of the metric space $(\mathcal{T}_N,\, \dtr)$ which we refer to as {\em palm tree space} (tropical tree space).  We show that palm tree space possesses fundamental characteristics for studies in probability and statistics to be well-defined; namely, that it is a Polish space.

\subsection{The Tropical Metric}

\begin{definition}
\label{def:tropicalmetric}
For two points $[x], [y] \in \mathbb{R}^n/\mathbb{R}\one$, consider the distance between $[x]$ and $[y]$ given by 
$$
\dtr([x], [y]) := \max_{1 \leq i < j \leq n} \big| (x_i - y_i) - (x_j - y_j) \big| = \max_{1 \leq i \leq n} (x_i - y_i) - \min_{1 \leq i \leq n}(x_i - y_i). 
$$
We refer to the function $\dtr$ as the {\em tropical metric}.
\end{definition}

\begin{proposition}
	The function $d_{\mathrm{tr}}$ is a well-defined metric function on $\mathbb{R}^{n}/\mathbb{R}\one$.
\end{proposition}
\begin{proof}
We verify that the defining properties of metrics are satisfied.  By definition, for $[u],[v] \in \mathbb{R}^{n}/\mathbb{R}\one$, $d_{\mathrm{tr}}([u], [v])= d_{\mathrm{tr}}([v], [u])$, satisfying symmetry.  The tropical metric is nonnegative, since $\big| (u_{i}-v_{i}) - (u_{j} - v_{j}) \big| \ge 0$, so is $d_{\mathrm{tr}}([u],[v])\geq 0$.  If $d_{\mathrm{tr}}([u],[v])=0$, then $u_{i}-v_{i}$ are equal for all $1\le i\le n$, thus $[u] = [v]$, so indiscernibles are identifiable.  

For $[u],[v],[w] \in \mathbb{R}^{n}/\mathbb{R}\one$, we now show that triangle inequality is satisfied: $d_{\mathrm{tr}}([u],[w])\le d_{\mathrm{tr}}([u],[v])+d_{\mathrm{tr}}([v],[w])$. Suppose $1\le i'<j'\le n$ such that
$$
\big|(u_{i'}-w_{i'})-(u_{j'}-w_{j'}) \big|=\max_{1\le i<j\le n}{|u_{i}-w_{i}-u_{j}+w_{j}|},
$$
then $d_{\mathrm{tr}}([u],[w])=|u_{i'}-w_{i'}-u_{j'}+w_{j'}|$. Note that
$
		u_{i'}-w_{i'}-u_{j'}+w_{j'}=(u_{i'}-v_{i'}-u_{j'}+v_{j'})+(v_{i'}-w_{i'}-v_{j'}+w_{j'}).
$
	Hence 
	\begin{align*}
		d_{\mathrm{tr}}([u],[w])=|u_{i'}-w_{i'}-u_{j'}+w_{j'}| & \le |u_{i'}-v_{i'}-u_{j'}+v_{j'}|+|v_{i'}-w_{i'}-v_{j'}+w_{j'}|\\
		& \le d_{\mathrm{tr}}([u],[v])+d_{\mathrm{tr}}([v],[w]).
	\end{align*}
Thus, $d_{\mathrm{tr}}$ is a well-defined metric function on $\mathbb{R}^{n}/\mathbb{R}\one$. 
\end{proof}

Notice that the metric space $(\mathbb{R}^{n}/\mathbb{R}\one, \dtr)$ can be identified with the normed linear space $\mathbb{R}^{n-1}$ via the linear isomorphism $\pi: \mathbb{R}^n/\mathbb{R}\one \rightarrow \mathbb{R}^{n-1}$ with $[x] \mapsto (x_2 - x_1,\, \ldots,\, x_n - x_1)$.  $\pi$ is in fact an isometry: define a norm on $\mathbb{R}^{n-1}$ by $\| x \|_{\mathrm{tr}} := \max(\max|x_i - x_j|,\, \max|x_i|)$ and denote the induced distance by $\hat{d}_{\mathrm{tr}}$, then
\begin{equation}
\label{eq:isometry}
\begin{aligned}
\dtr([x],[y]) & = \max\bigg(\max_{2 \leq i < j \leq n}|(x_i - y_i) - (x_j - y_j)|,\, \max_{2\leq i \leq n}|(x_i - x_1) - (y_i - y_1)| \bigg)\\
& = \| \pi([x]) - \pi([y])\|_{\mathrm{tr}} = \hat{d}_{\mathrm{tr}}(\pi([x]), \pi([y])).
\end{aligned}
\end{equation}
by choosing representatives such that $x_1 = y_1$, allowing us to drop indices and simplify the expression.  Note that $\| \cdot \|_{\mathrm{tr}}$ is a well-defined norm, since we may always add an extra 0-valued coordinate to each $x$ to obtain an element $[x]$ in the tropical projective torus.  The norm $\|x\|_{\mathrm{tr}}$ is then $\dtr([x],0)$ and the triangle inequality is satisfied; homogeneity under usual scalar multiplication and positive definiteness are also satisfied.

Restricting to the subspace of phylogenetic trees equipped with the tropical metric gives the following construction.

\begin{definition}
\label{def:palm}
For a positive integer $N$, let $\mathcal{T}_{N}$ be the space of phylogenetic trees with $N$ leaves.  The metric space $\mathcal{P}_N := (\mathcal{T}_{N}, d_{\mathrm{tr}} )$ is called the {\em palm tree space}.
\end{definition}

The spaces $\mathcal{T}_N^{\mathrm{BHV}}$ and $\mathcal{P}_N$ are not isometric, meaning that absolute lengths measured by each metric are not consistent.  To understand the variation in length discrepancy, we study the stability of the tropical metric $d_{\mathrm{tr}}$ and find that perturbations of points in BHV space, measured by the BHV metric $d_{\mathrm{BHV}}$, correspond to bounded perturbations of their images in palm tree space, measured by the tropical metric.  This stability property is desirable, since it allows for interpretable comparisons between the two spaces, and allows for ``translations'' in the widely-used BHV framework over to palm tree space.

The following lemma ensures coordinate-wise stability of the tropical metric in $\mathcal{P}_N$.

\begin{lemma}\label{lem:stable}
	Let $u\in \mathbb{R}^{n}$.  For $1\le i\le n$, if we perturb the $i$th coordinate of $u$ by $\varepsilon > 0$ to obtain another point $u'\in \mathbb{R}^{n}$, then in $\mathbb{R}^{n}/\mathbb{R}\one$ we have 
$
		d_{\mathrm{tr}}([u],[u'])=\varepsilon.
$
\end{lemma}

\begin{proof}
	For $1\le j\le n$, the difference $u'_{j}-u_{j}=0$ if $j\ne i$, and $u'_{i}-u_{i}=\pm \varepsilon$.  The set of these differences is then either $\{0,\varepsilon\}$ or $\{0,-\varepsilon \}$.  By Definition \ref{def:tropicalmetric}, $d_{\mathrm{tr}}([u],[u'])=|0-\pm \varepsilon|=|\varepsilon|$.
\end{proof}

\begin{theorem}[Stability]
\label{thm:stability}
Let $N$ be the number of leaves for phylogenetic trees in palm tree space and BHV space.  Let $u$ and $u'$ be two phylogenetic trees with $N$ leaves.  Then 
$$
d_{\mathrm{tr}}(u, u') \leq \sqrt{N+1} \cdot d_{\mathrm{BHV}}(u, u').
$$
Moreover, the smallest possible constant is $\sqrt{N+1}$.
\end{theorem}

\begin{proof}
	We first prove that for any two trees $u,u'$ in vector representation (\ref{eq:W}) with $N$ leaves, $d_{\mathrm{tr}}(u, u') \leq \sqrt{N+1}\cdot d_{\mathrm{BHV}}(u,u')$.  First, assume that $u,u'$ belong to the same orthant in BHV space.  Then no matter what the tree topology is, if we denote the differences of the lengths of the $N-2$ internal edges in $u$ and $u'$ (see (\ref{eq:link})) by $d_{1},d_{2},\ldots,d_{N-2}$, and the differences of the length of the $N$ external edges by $p_{1},p_{2},\ldots,p_{N}$, we always have 
$
		d_{\mathrm{BHV}}(u,u') = \sqrt{\sum_{i=1}^{N-2}{d_{i}^2}+\sum_{i=1}^{N}{p_{i}^2}}.
$
	
	For every pair of leaves $i,j$ in both trees, the distance between them is a sum of the length of some internal edges and two external edges.  In other words, all differences $w^{u}_{ij}-w^{u'}_{ij}$ are of the form of the sum between some $d_{k}$, and $p_{i}+p_{j}$.  Thus, the maximum of these differences is at most the sum of all positive $d_{i}$ values, plus the two greatest $p_{i}$ values (take these to be $p_{i_{1}}$ and $p_{i_{2}}$), while the minimum of these differences is at least the sum of all negative $d_{i}$ values, plus two smallest $p_{i}$ values (take these to be $p_{i_{3}}$ and $p_{i_{4}}$).  By definition, $d_{\mathrm{tr}}(u,u')$ is the maximum minus the minimum of these differences, so we have
	\begin{equation*}
		d_{\mathrm{tr}}(u,u') \le \sum_{i=1}^{N-2}{|d_{i}|}+|p_{i_{1}}|+|p_{i_{2}}|+|p_{i_{3}}|+|p_{i_{4}}|.
	\end{equation*}
	By the Cauchy--Schwarz inequality,
	\begin{equation*}
		(N+1)\cdot \Bigg(\sum_{i=1}^{N-2}{|d_{i}|^{2}}+|p_{i_{1}}|^{2}+|p_{i_{2}}|^{2}+|p_{i_{3}}|^{2}+|p_{i_{4}}|^{2} \Bigg) \ge \Bigg(\sum_{i=1}^{N-2}{|d_{i}|}+{|p_{i1}|+|p_{i_{2}}|+|p_{i_{3}}|+|p_{i_{4}}|} \Bigg)^2.
	\end{equation*}
Hence
	\begin{equation*}
		\begin{split}
			d_{\mathrm{tr}}(u,u') & \le \sum_{i=1}^{N-2}{|d_{i}|}+|p_{i_{1}}|+|p_{i_{2}}|+|p_{i_{3}}|+|p_{i_{4}}| \\
			& \le \sqrt{N+1}\cdot \sqrt{\sum_{i=1}^{N-2}{|d_{i}|^{2}}+|p_{i_{1}}|^{2}+|p_{i_{2}}|^{2}+|p_{i_{3}}|^{2}+|p_{i_{4}}|^{2}} \\
			& \le \sqrt{N+1}\cdot \Bigg(\sum_{i=1}^{N-2}{|d_{i}|^{2}}+\sum_{i=1}^{N}{p_{i}^2} \Bigg) \\
			& = \sqrt{N+1}\cdot d_{\mathrm{BHV}}(u,u').
		\end{split}
	\end{equation*}
	
	Now, for $u,u'$ with distinct tree topologies, we consider the unique geodesic connecting them: there exist finitely many points $u^{1},\ldots,u^{k-1}$ in BHV space such that $u^{i}$ and $u^{i+1}$ belong to the same orthant corresponding to a tree topology for $0\le i\le k-1$, where $u^{0}=u$ and $u^{k}=u'$, and $d_{\mathrm{BHV}}(u,u')=\sum_{i=0}^{k-1}{d_{\mathrm{BHV}}(u^{i},u^{i+1})}$. For $1\le i\le k-1$, by the proof above, we have that
$
		d_{\mathrm{tr}}(u^{i},u^{i+1}) \le \sqrt{N+1}\cdot d_{\mathrm{BHV}}(u^{i},u^{i+1}) \quad \forall\quad 1\le i\le k-1.
$
	Thus,
	$$
		d_{\mathrm{tr}}(u,u')\le \sum_{i=0}^{k-1}{d_{\mathrm{tr}}(u^{i},u^{i+1})} \le \sum_{i=0}^{k-1}{\sqrt{N+1}\cdot d_{\mathrm{BHV}}(u^{i},u^{i+1})}
		=\sqrt{N+1}\cdot d_{\mathrm{BHV}}(u,u').
	$$
	
	Next, we consider the case where the equality holds: consider two trees $t$ and $t'$ with $N$ leaves and the same tree topology, given by the following nested sets
$$
	\big\{\{1,2\},\{1,2,3\},\ldots,\{1,2,\ldots,N-2\} \big\}.
$$
	Let $e_i$ be the internal edge labeled by the $i$th clade in the nested sets listed previously.  Suppose in $t$, the internal edges have lengths
$$
	b^{t}(e_i) = \begin{cases}
	2, & \text{ if } 1\le i\le N-4; \\
	1, & \text{ if } i=N-3.
	\end{cases}
$$	
	Similarly, in $t'$, the internal edges have lengths
$$
	b^{t'}(e_i) = \begin{cases}
	1, & \text{ if } 1\le i\le N-4; \\
	2, & \text{ if } i=N-3.
	\end{cases}
$$
	The external edge lengths of $t$ and $t'$ are
$$	
	p^{i}_{j}=\begin{cases}
	1, & \text{ if } (i,j) = (1,2), (1,N-2), (2,N-1), (2,N); \\
	0, & \text{ otherwise}.
	\end{cases}
$$
Then 
$
		d_{\mathrm{BHV}}(t,t')=\sqrt{(N-4)\cdot (2-1)^{2}+(1-2)^2+2\cdot (1-0)^{2}+2\cdot (0-1)^{2}}=\sqrt{N+1}.
$
For $1\le i<j \le N$, in either tree the distance $w_{ij}$ is the sum of the edge lengths of 
$$
		p_{i},\, e_{\max(i-1,1)},\, e_{\max(i-1,1)+1},\, \ldots,\, e_{\max(j-2,N-2)},\, p_{j}.
$$
	Since $b^{t}(e_i)>b^{t'}(e_i)$ for $i<N-3$ and $b^{1}(e_i)<b^{2}(e_i)$ for $i=N-3$, the maximum of all differences $w^{t}_{ij}-w^{t'}_{ij}$ is 
$
	w^{t}_{2(N-2)} - w^{t'}_{2(N-2)} = ((N-4)\cdot 2 +2\cdot 1) - (N-4)\cdot 1 = N-2;
$
	and the minimum of all differences $w^{t}_{ij}-w^{t'}_{ij}$ is 
$
	w^{t}_{(N-2)(N-1)} - w^{t'}_{(N-2)(N-1)} = 1 - (2+1+1) = -3.
$
	By definition, $d_{\mathrm{tr}}(t,t')=(N-2)-(-3)=N+1=\sqrt{N+1}\cdot d_{\mathrm{BHV}}(t,t')$ in this case.  Thus, $\sqrt{N+1}$ is the smallest possible stability constant.  
\end{proof}

In general, and especially data applications, the number of leaves is fixed prior to the study so the stability constant $\sqrt{N+1}$ is indeed a constant.

We note that explicit calculations involving geodesics between trees in the original reference by \cite{BILLERA2001733} do not include external edges, since these do not modify the geometry of the space.  Indeed, their inclusion only amounts to an additional Euclidean factor, since the tree space then becomes the cross product of BHV space of trees with internal edges only, and $\mathbb{R}_{\geq 0}^N$.  Geodesic distances, which depend directly on geodesic paths (the former is the length of the latter), considered in \cite{BILLERA2001733} also do not include external edges.  In the proof of Theorem \ref{thm:stability}, we follow the quartic-time algorithm of \cite{Owen:2011:FAC:1916480.1916603} to compute BHV distances which includes external edge lengths, not only because it is the fastest algorithm to date but also necessary in this comparative setting, since the tropical distance is defined by external edge lengths.

In terms of interpretation, Theorem \ref{thm:stability} provides an important comparative measure and guarantees that quantitative results from BHV space are bounded in palm tree space.  For example, in single-linkage clustering, where clusters are fully determined by distance thresholds, the stability result means that a given clustering pattern in BHV space will be preserved in palm tree space, thus maintaining interpretability of clustering behavior.

\begin{figure}
\centering
 \begin{tikzpicture}[scale=0.4]
    	\draw (0,0) -- (3,5) -- (5,2);
    	\filldraw [black] (0,0) circle (1pt);
\filldraw [black] (5,2) circle (1pt);
    	\node [below] at (0,0) {$34$};
	\node [below] at (5,2) {$12$};
    	\node at (4.75,3.5) {$k_2$};
	\node at (0.75,2.5) {$k_1$};
    	\node at (0,4.5) {$T_{1}$};
    	
    	\draw (10,0) -- (13,5) -- (15,2);
    	\filldraw [black] (10,0) circle (1pt);
	\filldraw [black] (15,2) circle (1pt);
    	\node [below] at (10,0) {$12$};
	\node [below] at (15,2) {$34$};
    	\node at (14.75,3.5) {$k_1$};	
	\node at (10.75,2.5) {$k_2$};
    	\node at (10.5,4.5) {$T_{2}$};   	
    \end{tikzpicture}
    
  \caption{An example for non-isometry.  $T_1$ and $T_2$. }\label{fig:noniso_ex} 
\end{figure}

Note also that the reverse inequality does not hold in palm tree space: there is no general bound for the BHV distance in terms of the tropical distance.  Consider trees in $\mathcal{T}_4^{\operatorname{BHV}}$ leaves depicted in Figure \ref{fig:BHV4}.  In particular, consider the upper righthand orthant defined by the axes $D_{34}$ and $D_{12}$.  In Figure \ref{fig:noniso_ex}, take a tree $T_1$ with a coordinate representing the length of the internal edge of the $\{12\}$ clade at a distance $k_1$ from the origin, while the coordinate representing the length of the internal edge of the $\{34\}$ clade is at distance $k_2$ from the origin.  Similarly, take a tree $T_2$ with the coordinate representing the length of the internal edge of the $\{1,2\}$ clade at a distance $k_2$ from the origin, while the coordinate representing the length of the internal edge of the $\{34\}$ clade is at distance $k_1$ from the origin.  Then $T_1$ can be vectorized to $(k_1, k_2)$ while $T_2$ can be vectorized to $(k_2, k_1)$ and $\dtr(T_1, T_2) = 0$, but $d_{\mathrm{BHV}}(T_1, T_2) = \sqrt{(k_1 - k_2)^2 + (k_1-k_2)^2} = \sqrt{2(k_1-k_2)^2} = \sqrt{2}(k_1 - k_2) > 0$.  This example also shows that the tropical metric and BHV metric are not isometric.

\subsection{Geometry of Palm Tree Space}

The uniqueness property of geodesics in BHV space, used in the proof of Theorem \ref{thm:stability}, leads naturally to the study of similar geometric properties that characterize palm tree space as well as important differences between the two spaces.  These characteristics will now be developed in this section.

\subsubsection{Geodesics in Palm Tree Space}
\label{subsec:geodesics}

In palm tree space, geodesics are in general not unique, which is a common occurrence in various metric spaces.  There exists, however, a unique path connecting two points in ultrametric tree space within palm tree space, which is also a geodesic: the tropical line segment.

\begin{definition}
\label{def:trop_line}
Given $[x], [y] \in \mathbb{R}^n/\mathbb{R}\one$, the {\em tropical line segment} with endpoints $[x]$ and $[y]$ is the set
$$
\{a \odot [x] \oplus b \odot [y] \in \mathbb{R}^n/\mathbb{R}\one \mid a, b \in \mathbb{R} \},
$$
where $\odot$ is tropical multiplication and tropical addition $\oplus$ for two vectors is performed coordinate-wise.
\end{definition}

\begin{proposition}
\label{prop:geodesic}
For two trees $t, t' \in \mathcal{U}_N$, the tropical line segment connecting $t$ and $t'$ is a geodesic.
\end{proposition}

\begin{proof}
It suffices to show that for any $a,b\in \mathbb{R}$, we have that
$
  	  d_{\mathrm{tr}}(z,t)+d_{\mathrm{tr}}(z,t') = d_{\mathrm{tr}}(t,t'),
$
where $z=a\odot t \oplus b\odot t'$ lies on the tropical line segment.  We may assume that $t_{i}-t'_{i}\leq t_{i+1}-t'_{i+1}$ for $1\le i \le n-1$.  Under this assumption, $d_{\mathrm{tr}}(t,t')=(t_{n}-t'_{n})-(t_{1}-t'_{1})$.  Now if $0\le j\le n$ is the largest index such that $t_{j}-t'_{j}\le b-a$, then for some $i\ge j+1$, $z_{i}=b+t'_{i}$ and for $i \leq j$, $z_{i}=a+t_{i}$.  If $j=0$ or $j=n$, then $z$ is equal to either $t$ or $t'$ and the claim is apparent.  We may thus assume $1\le j\le n-1$. 
 
  The set of all differences $t_{i}-z_{i}$ contains $-a$ and the greater values $t_{i}-t'_{i}-b > -a$ for $i\ge j+1$.  So,
$
  	d_{\mathrm{tr}}(z,t)=(t_{n}-t'_{n}-b)-(-a)=(t_{n}-t'_{n})+(a-b).
$
Similarly, the set of all differences $z_{i}-t'_{i}$ contains $b$ and the smaller values $(t_{i}-t'_{i})+a \leq b$ for $i\le j$.  So,
$
  d_{\mathrm{tr}}(z,t')=b-(t_{1}-t'_{1}+a)=(b-a)-(t_{1}-t'_{1}).
$
  Therefore, $d_{\mathrm{tr}}(z,t)+d_{\mathrm{tr}}(z,t') = d_{\mathrm{tr}}(t,t'),$ and the tropical line segment connecting $t$ and $t'$ is a geodesic.
\end{proof}

In addition, it turns out that tropical line segments are easy and fast to compute.  In particular, the time complexity to compute them is lower than the state-of-the-art in BHV space \citep{Owen:2011:FAC:1916480.1916603}.

\begin{proposition}{\cite[Proposition 5.2.5]{maclagan2015introduction}}
The time complexity to compute the tropical line segment connecting two points in $\mathbb{R}^n/\mathbb{R}\one$ is $O(n\, \log\, n) = O(N^2\, \log\, N)$.
\end{proposition}

\subsubsection{Structure of Palm Tree Space}

In the same way that $\mathcal{T}_N^{\mathrm{BHV}}$ is constructed as the union of orthants, the geometry of $\mathcal{P}_N$ is also given by such a union.

\begin{proposition}\cite[Proposition 4.3.10]{maclagan2015introduction}
\label{prop:polyhedra}
The space $\mathcal{T}_{N}$ is the union of $(2N-5)!!$ polyhedra in $\mathbb{R}^n/\mathbb{R}\one$, each of dimension $N-3$.
\end{proposition}

\subsection{Topology of Palm Tree Space}

The measure of a space is relevant in probabilistic studies, since the topology of a space may be interpreted in terms of measures.  For example, Radon measures may also be interpreted as linear functionals on the space of continuous functions with compact support, which is locally convex, by e.g.,~\cite[Chapter 3]{bourbakielements}.  This motivates our study of the topology of palm tree space.

The following two lemmas allow us to characterize the topology of palm tree space.  Recall that for $x\in \mathbb{R}^{n}$, the set $B(x,r)=\{y\in \mathbb{R}^{n} \mid |y-x|<r\}$, with $|\cdot|$ taken to be the Euclidean norm, is the open ball centered at $x$ with radius $r$.  By identifying $\mathbb{R}^{n}/\mathbb{R}\one$ with $\mathbb{R}^{n-1}$ via (\ref{eq:isometry}), an equivalent set may be correspondingly defined in palm tree space by considering the tropical norm as follows.

\begin{definition}\label{def:openball}
Under the tropical metric $d_{\mathrm{tr}}$, we define
$$
B_{\mathrm{tr}}(x,r)=\{y\in \mathbb{R}^{n} \mid d_{\mathrm{tr}}((0, y),\, (0, x) )<r \}
$$
to be the open {\em tropical ball} centered at $x \in \mathbb{R}^{n}$ with radius $r$.
\end{definition}

\begin{lemma}\label{lem:unitball}
	For $x,y \in \mathbb{R}^{n-1}$ and $r>0$, the open tropical ball $B_{\mathrm{tr}}(x,r)$ is the open convex polytope defined by the following strict inequalities for $1\le i<j\le n-1$:
		\begin{equation}\label{eq:ball}
		\begin{split}
			y_{i} & >x_{i}-r, \\
			y_{i} & <x_{i}+r, \\
			y_{i}-y_{j} & >x_{i}-x_{j}-r, \\
			y_{i}-y_{j} & <x_{i}-x_{j}+r.\\
		\end{split}
	\end{equation} 
\end{lemma}

\begin{proof}
For $y\in \mathbb{R}^{n-1}$, $y\in B_{\mathrm{tr}}(x,r)$ if and only if $d_{\mathrm{tr}}((0,x),(0,y) )<r$.  Definition \ref{def:tropicalmetric} admits the strict inequalities in (\ref{eq:ball}).
\end{proof}

\begin{lemma}\label{lem:contain}
For $r>0$ and $x\in \mathbb{R}^{n-1}$, $B(x,r) \subseteq B_{\mathrm{tr}}(x,2r)$ and $ B_{\mathrm{tr}}(x,r) \subseteq B(x,\sqrt{n-1}r)$.
\end{lemma}

\begin{proof}
By Lemma \ref{lem:unitball}, if a point $y$ lies in $B_{\mathrm{tr}}(x,r)$, then for $1\le i\le n-1$, $|y_{i}-x_{i}|<r$, thus $y\in B(x,\sqrt{n-1}r)$.  Conversely, if a point $y$ lies in $B(x,r)$, then for $1\le i\le n-1$, we have that $|y_{i}-x_{i}|<r$. Therefore, 
	\begin{equation*}
	\big|(y_{i}-y_{j})-(x_{i}-x_{j}) \big| = \big|(y_{i}-x_{i})-(y_{j}-x_{j}) \big|<2r.
	\end{equation*}
Hence $y\in B_{\mathrm{tr}}(x,2r)$.
\end{proof}

\begin{theorem}\label{thm:openset}
On $\mathbb{R}^{n-1}$, the family of open balls $B(x,r)$ and the family of open tropical balls $B_{\mathrm{tr}}(x,r)$ define the same topology.
\end{theorem}

\begin{proof}
Suppose for all $r>0$ and $x\in \mathbb{R}^{n-1}$ that the open balls $B(x,r)$ form a topological basis.  For any $y\in \mathbb{R}^{n-1}$ and $s>0$, we consider the ball $B_{\mathrm{tr}}(y,s)$: For any point $z\in B_{\mathrm{tr}}(y,s)$, we have that $d_{\mathrm{tr}}(z,y)<s$.  Let $\displaystyle \varepsilon=\frac{s-d_{\mathrm{tr}}(z,y)}{2}>0$.  Then $B_{\mathrm{tr}}(z,2\varepsilon)\subseteq B_{\mathrm{tr}}(y,s)$.  By Lemma \ref{lem:contain}, we have $B(z,\varepsilon)\subseteq B_{\mathrm{tr}}(z,2\varepsilon) \subseteq B_{\mathrm{tr}}(y,s)$.  Therefore, $B_{\mathrm{tr}}(y,s)$ is also an open set.  The other direction is proved in the same manner.
\end{proof}

\begin{example}\label{ex:balls}	
Figure \ref{fig:balls} illustrates the unit balls in Euclidean, BHV, and palm tree space.  Here, the number of leaves is fixed to be $3$.  There are three $1$-dimensional cones in BHV space, and they share the origin.  The palm tree space $\mathcal{P}_3 = \{w=(w_{12},w_{13},w_{23})\in \mathbb{R}^{3}/\mathbb{R}\one \mid \max(w) \text{ is attained at least twice} \}$ may be embedded in $\mathbb{R}^{2}$.

	\begin{figure}[h]
		\centering
		\begin{minipage}{0.3\textwidth}
			\centering
			\begin{tikzpicture}
				\draw [fill=yellow,thick] (0,0) circle (1cm);
				\filldraw [black] (0,0) circle (2pt);
				\node [above] at (0,0) {$(0,0)$};
			\end{tikzpicture}
		\end{minipage}
		\begin{minipage}{0.3\textwidth}
			\centering
			\begin{tikzpicture}[scale=0.4]
			\draw (0,3) -- (0,0) -- (3,0);
			\draw (0,0) -- (-1.5,-2);
			\draw [fill=yellow,ultra thick] (0,1) -- (0,0);
			\draw [fill=yellow,ultra thick] (0,0) -- (1,0);
			\draw [fill=yellow,ultra thick] (-0.6,-0.8) -- (0,0);
			\filldraw [black] (0,0) circle (2pt);
			\node [above right] at (0,0) {$(0,0,0)$};
			\node [right] at (0,3) {$\{1,2\}$};
			\node [right] at (3,0) {$\{1,3\}$};
			\node [left] at (-1,-2) {$\{2,3\}$};
			\end{tikzpicture}
		\end{minipage}
		\begin{minipage}{0.3\textwidth}
			\centering
			\begin{tikzpicture}[scale=0.8]
				\draw [fill=yellow,thick] (1,0) -- (1,1) -- (0,1) -- (-1,0) -- (-1,-1) -- (0,-1) -- (1,0);
				\filldraw [black] (0,0) circle (2pt);
				\node [above right] at (0,0) {$[(0,0,0)]$};
			\end{tikzpicture}
		\end{minipage}
		\caption{Comparison of unit balls in Euclidean, BHV, and palm tree space for $N=3$ leaves.  The leftmost figure is the unit ball $B((0,0),\, 1)$ in $\mathbb{R}^2$; the center figure is the unit ball centered at the origin with radius 1 in a BHV space with 3 leaves; the rightmost figure is the unit ball $B_{\mathrm{tr}}([(0,0,0)],\, 1)$ in $\mathcal{P}_{3}$.}
		\label{fig:balls}
	\end{figure}
\end{example}

\subsection{Palm Tree Space is a Polish Space}
\label{subsec:polish}

We now show that additional analytic properties of palm tree space that are desirable for probabilistic and statistical analysis are satisfied.  Specifically, we prove that palm tree space is a separable, completely metrizable topological space, and thus a Polish space, by definition.

Polish spaces are important settings for studies in probability due to the fact that classical results may be formulated and generalized in a well-behaved manner; some examples are the construction of conditional expectations, Kolmogorov's extension theorem (which guarantees the definition of a stochastic process from a series of finite-dimensional distributions), and Prokhorov's theorem (which guarantees weak convergence by relating tightness of measures to compactness in a probability space) \citep{1967}.

\begin{proposition}
$\mathcal{P}_{N}$ is complete.
\end{proposition}

\begin{proof}
For convenience, when considering points in $\mathcal{P}_{N}$, we always choose their unique preimage in $\mathbb{R}^{n}$ whose first coordinate is $0$.  Then, we may denote each point in $\mathcal{P}_{N}$ by an $(n-1)$-tuple in $\mathbb{R}^{n-1}$.  Let $t_{1},t_{2},\ldots \in \mathbb{R}^{n-1}$ be a Cauchy sequence of points in $\mathcal{P}_{N}$.  For $1\le i\le n-1$, we claim that $(t_{k}^{i})_{k\ge 1}$ also form a Cauchy sequence in $\mathbb{R}$: For any $\varepsilon>0$, there exists $M$ such that for $k_{1},k_{2}>M$, we have $d_{\mathrm{tr}}(t_{k_{1}},t_{k_{2}})<\varepsilon$.  By Definition \ref{def:tropicalmetric}, $d_{\mathrm{tr}}(t_{k_{1}},t_{k_{2}}) \ge |0-0-t_{k_{2}}^{i}+_{k_{1}}^{i}|=|t_{k_{1}}^{i}-t_{k_{2}}^{i}|$.  Thus, for $k_{1},k_{2}>M$, we have 
$
		\big|t_{k_{1}}^{i}-t_{k_{2}}^{i} \big|<\varepsilon.
$
	
Suppose now that the Cauchy sequence $(t_{k}^{i})_{k\ge 1}$ converges in the Euclidean metric to $t_{0}^{i} \in \mathbb{R}$.  It suffices to show
	\begin{itemize}
		\item[(i)] $t_{0}=(t_{0}^{1},t_{0}^{2},\ldots,t_{0}^{n-1})$ represents a point in $\mathcal{P}_{N}$;
		\item[(ii)] $\lim\limits_{k\to \infty}{d_{\mathrm{tr}}(t_{k},t_{0})}=0$.
	\end{itemize}
	To show (ii), we argue that since $(t_{k}^{i})_{k\ge 1}$ converges to $t_{0}^{i}$ for all $1\le i\le n-1$, then for any $\varepsilon>0$ there exists $M$ such that for $k > M$, we have $|t_{k}^{i}-t_{0}^{i}|<\frac{\varepsilon}{2}$ for all $1\le i\le n-1$.  Then by Definition \ref{def:tropicalmetric},
$$
			d_{\mathrm{tr}}(t_{k},t_{0})=\max_{1\le i\le n-1}{\big(0,t_{k}^{i}-t_{0}^{i} \big)}-\min_{1\le i\le n-1}{\big(0,t_{k}^{i}-t_{0}^{i}\big)}
			<\frac{\varepsilon}{2}-\Big(-\frac{\varepsilon}{2} \Big)=\varepsilon.
$$
	So $\lim\limits_{k\to \infty}{d_{\mathrm{tr}}(t_{k}-t_{0})}=0$.
	
	To show (i), note that each coordinate of $t_{0}$, including the first, is the limit of the corresponding coordinates of $(t_{k})_{k\ge 1}$.  Suppose $t_{0} \notin \mathcal{P}_{N}$, then there exists $1\le i<j<k<l\le N$ such that one term of $t_{0}$ in (\ref{eq:plucker}) is strictly greater than the remaining two.  Then there exists $M_{2}$ such that for all $k>M_{2}$, the one term of $t_{k}$ in (\ref{eq:plucker}) is also strictly greater than the remaining two, thus $t_{k} \notin \mathcal{P}_{N}$---a contradiction.  Hence (i) holds, and $\mathcal{P}_{N}$ is complete.
\end{proof}

\begin{proposition}
$\mathcal{P}_{N}$ is separable.
\end{proposition}

\begin{proof}
We claim that the set of all trees with all rational coordinates is dense in $\mathcal{P}_{N}$: Fix any tree $t=(w_{ij})\in \mathcal{P}_{N}$.  By Proposition \ref{prop:polyhedra}, $t$ belongs to a polyhedron and there exists a tree topology with $(N-3)$ internal edges.  Then the distance between any two leaves is the sum of the lengths of the edges along the unique path connecting them.  The number of edges along each path is at most $(N-1)$.  For any $\varepsilon>0$ and length $b_{k}$ of each edge of the tree $t$, since $\mathbb{Q}$ is dense in $\mathbb{R}$, we can find a rational number $q_{k}$ such that $|q_{k}-b_{k}|<\frac{1}{2(N-1)}\varepsilon$.  Now, construct another tree $t' = (w'_{ij})$ with the same topology as $t$, and with corresponding edge lengths $q_{k}$.  Then for any $1\le i<j\le N$ we have that $|w'_{ij}-w_{ij}|<\frac{\varepsilon}{2}$.
	Thus 
	\begin{equation*}
		d_{\mathrm{tr}}(t',t)=\max_{1\le i<j\le n}{(w'_{ij}-w_{ij})}-\min_{1\le i<j\le n}{(w'_{ij}-w_{ij})}<\varepsilon,
	\end{equation*}
and all coordinates of $q_{k}$ are rational.  Thus, $\mathcal{P}_N$ is separable.
\end{proof}

The above results on completeness and separability are proved by definition.  An alternative perspective that demonstrates completeness and separability is to identify the tropical projective torus and its corresponding tropical metric with $\mathbb{R}^n$ and the $\ell_\infty$ distance.  This identification has been used by \cite{ardila2005subdominant} and \cite{bernstein2017infinity, bernstein2020infinity}: Consider a linear mapping from $\mathbb{R}^N$ to $\mathbb{R}^n$ where $(x_1,\ldots,x_N) \mapsto (x_i - x_j)$ for all pairs $i < j$.  The image of such a map is isomorphic to the tropical projective torus and the tropical metric is then the $\ell_\infty$ distance on $\mathbb{R}^n$ restricted to the image of this map.  Palm tree space forms a closed subset of $\mathbb{R}^n$, since the four-point condition (Proposition \ref{prop:4pt}) defines a closed subset and $\mathbb{R}^n$ equipped with the $\ell_\infty$ distance is complete and separable.  This formulation also provides insight into the topology of palm tree space described in Theorem \ref{thm:openset}.  

Finally, compact subsets in palm tree space exist.  As an example, consider the space of ultrametric trees $\mathcal{U}_N$.  Let {\em compact tree space} $\mathcal{U}_N^{[1]}$ to be the image of $\mathcal{U}_N$ of the set of ultrametrics $W$ satisfying $\max_{i,j}(w(i,j)) = 1$.  Now, notice that the union of $\mathcal{U}_N^{[k]}$ for $1 \leq k \leq N$ is still compact and a subset is bounded if and only if it is contained in this union; in particular, if it is also closed, then it is compact.






\subsection{Probability Measures and Means in Palm Tree Space}
\label{sec:prob}

We showed in Section \ref{subsec:polish} that palm tree space is a Polish space, and thus exhibits desirable properties for rigorous probability and statistics.  Such properties ensure well-behaved measure-theoretic properties, and in particular, allow for classical probabilistic and statistical studies, such as convergence in various modes, as well as ensuring that stochastic processes are well defined.  We now study the existence of probabilistic and statistical quantities for parametric data analysis, such as probability measures and Fr\'echet means and variances.

\subsubsection{Tropical Measures of Central Tendency}
\label{subsec:means}

For distributions in general metric spaces, there are various measures of central tendency.  These may be framed in palm tree space as follows (and may be generalized by replacing the tropical metric $d_{\mathrm{tr}}$ with any well-defined metric).

\begin{definition}
Given a probability space $(\mathcal{T}_N, \mathcal B(\mathcal{T}_N), \mathbb{P}_{\mathcal{T}_N})$---where $\mathcal{T}_N$ is the set of all possible outcomes, $\mathcal B(\mathcal{T}_N)$ is the event space or set of outcomes in $\mathcal{T}_N$ (taken here to be the $\sigma$-algebra generated by the open tropical balls $B_{\mathrm{tr}}$ of $\mathcal{T}_N$), and $\mathbb{P}_{\mathcal{T}_N})$ is a probability function that assigns a proability to each event in the event space---the quantity
\begin{equation}\label{eq:frechet_variance}
\text{Var}_{\mathbb{P}_{\mathcal{T}_N}}(t) = \int_{\mathcal{T}_N} d_{\mathrm{tr}}(t, t')^2 d\nu(t') < \infty
\end{equation}
is known as the {\em tropical Fr\'echet variance}.  The minimizer of the quantity (\ref{eq:frechet_variance}) is the {\em tropical Fr\'echet population mean} or {\em barycenter} $\mu_F$ of a distribution $\nu$:
\begin{equation}\label{eq:frechet}
\mu^F_{\mathrm{tr}} = \arg\min_{t} \int_{\mathcal{T}_N} d_{\mathrm{tr}}(t, t')^2 d\nu(t') < \infty.
\end{equation}
\end{definition}


For general metric spaces, neither existence nor uniqueness of (\ref{eq:frechet}) 
is guaranteed.  The following condition ensures the existence of barycenters \citep{Ohta2012} .

\begin{lemma}{\cite[Lemma 3.2]{Ohta2012}}\label{lemma:frechet_existence}
If $(M, d)$ is a proper metric space---that is, a metric space where every closed, bounded subspace is compact---then any distribution $\nu$ where $\int_M d(t, t')^2 d\nu(t') < \infty$ has a barycenter.
\end{lemma}

Palm tree space is a proper metric space, since in order for every subspace to be bounded, the height of the tree must be fixed which automatically gives a compactification of the subspace as well.  Thus, (\ref{eq:frechet}) evaluated according to the tropical metric is guaranteed to exist.  However, since geodesics are not unique in palm tree space, Fr\'echet means will also, in general, not be unique.  


\subsubsection{Tropical Probability Measures}
\label{subsec:prob_measures}

Probability measures on combinatorial and phylogenetic trees have been previously discussed, for example by \cite{10.1007/978-1-4612-0719-1_1} and \cite{BILLERA2001733}.  This section is dedicated to an analogous discussion on palm tree space.  In $\mathcal{P}_N$, the Borel $\sigma$-algebra $\mathcal B(\mathcal{T}_N)$ is the $\sigma$-algebra generated by the open tropical balls $B_{\mathrm{tr}}$ of $\mathcal{T}_N$, given in Definition \ref{def:openball}.  We begin by providing the existence of probability measures on $\mathcal{P}_N$.

\begin{definition}
A {\em finite tropical Borel measure} on $\mathcal{T}_N$ is a map $\mu : \mathcal B(\mathcal{T}_N) \rightarrow [0, \infty)$ such that $\mu(\emptyset) = 0$, and for mutually disjoint Borel sets $A_1, A_2, \ldots \in \mathcal B(\mathcal{T}_N)$ implies that $\mu( \bigcup_{i=1}^\infty B_{\mathrm{tr}}^{(i)}) = \sum_{i=1}^\infty \mu( B_{\mathrm{tr}}^{(i)})$.  If in addition $\mu(\mathcal{T}_N) = 1$, then $\mu$ is a {\em tropical Borel probability measure} on $\mathcal{T}_N$.
\end{definition}

Since $\mathcal{T}_N$ is a finite union of polyhedra in $\mathbb{R}^{n}/ \mathbb{R}\one$ (see Proposition \ref{prop:polyhedra}), tropical Borel probability measures exist if finite tropical Borel measures $\mu$ exist on $\mathcal{T}_N$ by an appropriate scaling of the value of $\mu$ on each polyhedron.

Alternatively, the existence of probability measures on $\mathcal{T}_N$ can be seen by considering an arbitrary probability space $(\Omega, \mathcal{F}, \mathbb{P})$ together with a measurable map $X: \Omega \rightarrow \mathcal{T}_N$.  Such maps exist, since we have shown in Section \ref{sec:palm} that $\mathcal{T}_N$ is a Polish space and thus $( \mathcal{T}_N, \mathcal B(\mathcal{T}_N) )$ is a standard measurable space (e.g., \cite{taylor2012introduction}).  The probability space $(\Omega, \mathcal{F}, \mathbb{P})$ is a measure space by assumption, thus also measurable.  In this case, the map $X$ is a random variable taking values in $\mathcal{T}_N$.  Then, $X$ induces a probability measure $\mathbb{P}_{\mathcal{T}_N}$ on $(\mathcal{T}_N, \mathcal{B}(\mathcal{T}_N) )$ by the pushforward measure $X_*\mathbb{P}$ of $\mathbb{P}$ under $X$, known as the {\em distribution}, for all Borel sets $A \in \mathcal{B}(\mathcal{T}_N)$:
\begin{equation*}\label{eq:prob_measure}
X_*\mathbb{P}(A) := \mathbb{P}( X^{-1}(A)) = \mathbb{P}( \{ \omega \in \Omega \mid X(\omega) \in A \} ).
\end{equation*}

\subsection{Statistical and Geometric Comparisons Between Tree Spaces}

We close this section with a summary of the properties discussed thus far in a comparative manner between palm tree space and BHV space.  With respect to the statistical and probabilistic properties, we note that barycenters and probability measures exist in both palm tree space, as we have shown above, and in BHV space \citep{garba2021information}.  The differences in geometric properties of both palm tree and BHV space in Table \ref{tab:compare}.  Here, we note that the {\em depth} of a geodesic is the largest codimension among all polyhedra traversed by the geodesic; specifically, for a point $x$ belonging to the interior of a polyhedron of dimension $k$, the depth of the point $x$ is $N-2-k$.  The high depth of a geodesic path in BHV space is conjectured to play a role in the stickiness of BHV Fréchet means \citep{doi:10.1137/16M1079841}.\\

\begin{table}
\centering
\begin{tabular}{lcc}
& Palm Tree & BHV\\
\hline
Geodesic Paths 
& Not Unique & Unique\\
Complexity for Geodesic Computation 
& $O(N^2 \log N)$ & $O(N^4)$\\
Dimension of Geodesic Triangles 
& $\leq$ 2 & Unbounded\\
Depth of Random Geodesic Path 
& Low & High\\
Number of Polyhedra 
& $(2N-5)!!$ & $(2N-3)!!$\\
\hline
\end{tabular}
\caption{Summary of Geometric Comparisons Between Palm Tree Space and BHV Space}
\label{tab:compare}
\end{table}


\section{Classification and Exploratory Data Analysis in Palm Tree Space: A Numerical Experiment and Real Data Application}
\label{sec:app}

We now give concrete demonstrations of both inferential and descriptive statistical tasks in palm tree space to showcase its viability.  Specifically, to demonstrate the viability of statistical inference in palm tree space, we propose and develop a tropical version of linear discriminant analysis and demonstrate the method on simulated data.  To demonstrate the viability of descriptive statistics in palm tree space, we study a real data application of the seasonal influenza virus.  Here, we moreover provide a performance comparison to the BHV case with these real, large-scale data.

In this section, we focus on the setting of rooted phylogenetic trees.  

\subsection{Descriptive Statistics: Dimension Reduction in Tree Spaces}

We now provide an example of a descriptive study in palm tree space by applying the tropical PCA method of \cite{Yoshida2018} to real data.  We compare the performance of tropical versus BHV principal component analysis (PCA) of the seasonal influenza virus by studying its diversity over twenty years of collected longitudinal data.


\subsubsection{Influenza Data}
\label{sec:data}

We study the influenza type A virus, which is an RNA virus that is classified by subtype according to the two proteins occurring on the surface of the virus: hemagglutinin (HA) and neuraminidase (NA).  Here, we focus on HA, which tends to be the most variable protein in genomic evolution, in terms of changing the antigenic make-up of surface proteins.  Such antigenic variability (known as {\em antigenic drift}) is an important driving factor behind vaccine failure.  We restrict our study to the subtype H3N2, which is becoming increasingly abundant, and a dominant factor studied in developing flu vaccines due to its recently increasing resistance to standard antiviral drugs \citep{NYT}.  It was also the cause of an epidemic due to vaccine failure in 2002-2003 \citep{centers2004preliminary}.

Genomic data for 1089 full length sequences of hemagglutinin (HA) for influenza A H3N2 from 1993 to 2017 in the state of New York were obtained from the GI-SAID EpiFlu\textsuperscript{TM} database (\url{www.gisaid.org}) and aligned with {\sc muscle} \citep{10.1093/nar/gkh340} using default settings.  HA sequences from each season were related to those of the preceding season.  We then applied {\em tree dimensionality reduction} \cite{zairis2016genomic} using temporal windows of 5 consecutive seasons to create 21 datasets: The influenza virus is assumed to emerge and evolve from a common ancestor: although the virus mutates each season and within each patient, each of these seasonal and patient-specific mutations can be traced back to a single virus \cite[e.g.,][]{10.1371/journal.pone.0005022}.  This evolutionary pattern is depicted in a large phylogenetic tree.  Tree dimensionality reduction produces a collection of smaller trees that faithfully represents the evolutionary behavior of the single large tree, and thus allows the evolutionary information to be treated as a dataset with multiple points akin to bootstrapping, rather than viewing the large tree as a single datum \citep{zairis2016genomic}.  Among the 21 datasets, the date of each dataset corresponds to the first season; for example, the dataset dated 2013 consists of 5-leaved trees where the leaves come from seasons 2013 through 2017.  Each tree in these datasets was constructed using the neighbor-joining method \citep{doi:10.1093/oxfordjournals.molbev.a040454} with Hamming distance.  Outliers were then removed from each season using {\sc kdetrees} \citep{10.1093/bioinformatics/btu258}.  On average, there were approximately 20,000 remaining trees in each dataset.  Finally, PCA was performed under the tropical metric \citep{Yoshida2018} and under the BHV metric \citep{doi:10.1093/biomet/asx047}.




\subsubsection{Principal Component Analysis in Tree Spaces}
\label{sec:treePCA}

PCA is a fundamental technique in descriptive and exploratory statistics that visualizes relationships within the data and reduces dimension.  As such, PCA has many important implications---for example, it projects to the subspace of the solution of $k$-means clustering \cite{ding2004k}---and may be interpreted in several different ways.  One interpretation may be seen as searching for a lower-dimensional plane that minimizes the sum of squared distances from the data points to the plane, and then finding an orthogonal projection from the data points onto the plane to visualize them.  This is the approach adapted by \cite{doi:10.1093/biomet/asx047,Yoshida2018}, which we implement here.


Respective adaptations of the above-mentioned interpretation of PCA to palm tree space and BHV space are given by \cite{doi:10.1093/biomet/asx047} and \cite{Yoshida2018}: convex, triangular regions---the tropical triangle and the {\em locus} of weighted Fr\'{e}chet means computed with respect to the BHV metric, respectively---define notions of second principal components in the respective spaces.  We refer the reader to these references for details and here, we apply these methods to the influenza virus data to explore and compare the performance of both methods in large-scale real data.

\subsubsection{Interpretation of Tree PCA}

In the tropical case, the second principal component represented by the tropical triangle---whose vertices are given by three ultrametric trees, and whose edges are given by the tropical line segments between them---divides into cells, which are determined by the tree topologies that an edge (tropical line segment) traverses.  Trees in the dataset are then projected into the cells corresponding to their topologies in the PCA visualization.  The simplest case in both BHV and tropical PCA is where all three vertices of the triangle are of the same tree topology, then there will only be one cell and all projections will be of the same topology; if two points are trees of the same tree topology, then every point on the tropical line segment connecting them will also be of the same tree topology.  An example can be seen in the 1993 data set, available at \url{https://github.com/antheamonod/FluPCA/Figures}.

A more interesting example can be seen in the 2008 dataset in Figure \ref{fig:trop}.  The tropical triangle divides into six cells; 
the tropical line segment forming the topmost edge of the triangle traverses five different tree topologies (indicated by the pink, green, black, blue, and red points along the topmost line segment), indicating that the three vertices of the triangle are quite different from one another in terms of tree topology.

\begin{figure}
    \centering
        \includegraphics[width=0.85\textwidth]{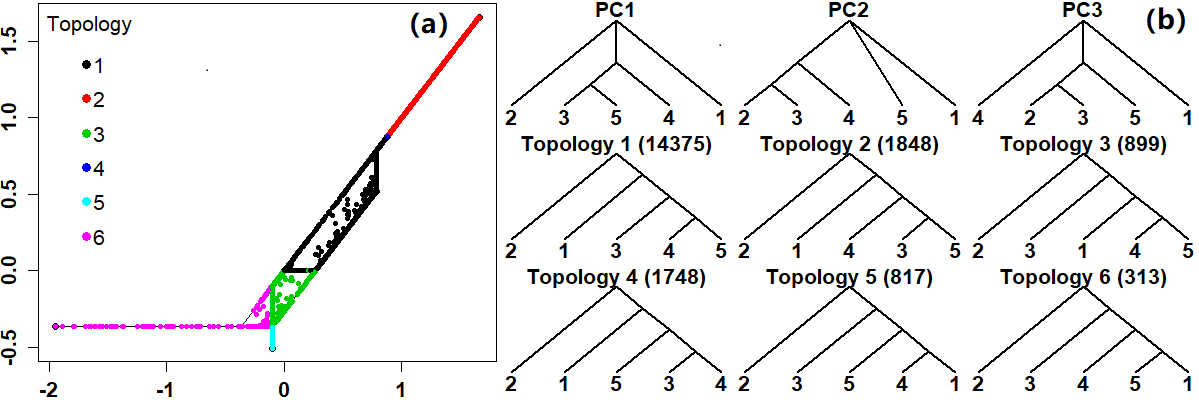}
    \caption{2008: The tropical triangle as the second tropical principal component. (a) Tropical triangle and projected data points; (b) Vertices of the tropical triangle and projected tree topologies, where the numbers appearing in parentheses are the frequencies for each tree topology.}
    \label{fig:trop}
\end{figure}

The BHV locus, which represents the second principal component in BHV space, is also generated by three vertices (ultrametric trees), and depicted in the BHV PCA plots by a triangle.  Here, varying tree topologies are depicted by the multicolored patches within the triangular region, indicating that the locus straddles several orthants in BHV space.  In the 2008 dataset in Figure \ref{fig:bhv}, the locus straddles two BHV orthants and we see two tree topologies occurring among the projected points.  

\begin{figure}
    \centering
        \includegraphics[width=0.85\textwidth]{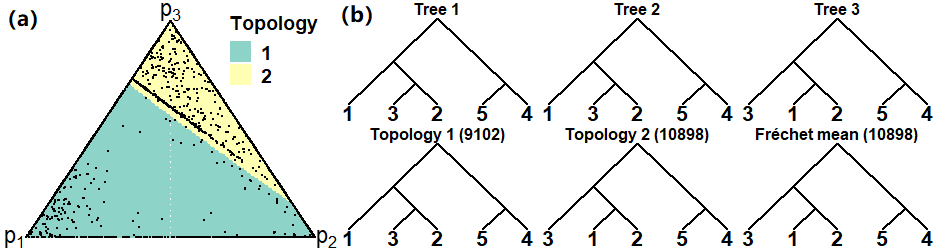}
    \caption{2008: The locus of BHV Fr\'{e}chet means as the second principal component. (a) The simplex shaded by topology of corresponding points on the affine subspace; (b) Trees 1, 2, and 3 correspond to three weighted Fr\'echet means, where the numbers appearing in parentheses are the frequencies for each tree topology.} 
    \label{fig:bhv}
\end{figure}

In the 2008 dataset, palm tree space and the tropical geometric approach to computations in phylogenetic tree space appears to allow for occurrences of richer and more subtle structures and methods: in these examples, we see tropical PCA projections of six different topologies, versus two in the BHV case.

\subsubsection{Results: Proportion of Standard Deviation Explained $R$}

In terms of standard deviation explained given in Table \ref{tab:var}, we see that in general, tropical PCA is able to explain more of the standard deviation in the data than does BHV PCA.  To compute the fraction of standard deviation explained, we used the approach of \cite{10.1093/bioinformatics/btaa564} as presented in Section 4.1.  BHV PCA results also have a higher variability over a wider range than tropical PCA: BHV PCA explains between 3.7\% and 97\% of the standard deviation, while tropical PCA explains between 46\% and 96\%.

\begin{table}
\centering
\caption{Proportion of Explained Standard Deviation ($R$) for Tropical and BHV PCA}
\label{tab:var}
\scalebox{0.85}{%
\begin{tabular}{cccccccc}
 & 1993 & 1994 & 1995 & 1996 & 1997 & 1998 & 1999 \\
\hline
Tropical & 0.7269 & 0.8505 & \textbf{0.9577} & 0.7482 & 0.8437 & 0.8790 & 0.8564 \\
BHV & 0.5496 & 0.6593 & 0.5613 & 0.7089 & 0.2247 & 0.8004 & \textbf{0.9760} \\
\hline
\end{tabular}}
\scalebox{0.85}{%
\begin{tabular}{cccccccc}
 & 2000 & 2001 & 2002 & 2003 & 2004 & 2005 & 2006 \\
\hline
Tropical & 0.7942 & 0.8302 & 0.9525 & 0.8622 & 0.7931 & 0.8304 & 0.7300 \\
BHV & \textbf{0.0374} & 0.9741 & 0.9467 & 0.7019 & 0.6042 & 0.6028 & 0.4882 \\
\hline
\end{tabular}}
\scalebox{0.85}{%
\begin{tabular}{cccccccc}
 & 2007 & 2008 & 2009 & 2010 & 2011 & 2012 & 2013 \\
\hline
Tropical & 0.6995 & \textbf{0.4637} & 0.6289 & 0.6665 & 0.5920 & 0.5568 & 0.5624 \\
BHV & 0.5222 & 0.2144 & 0.3953 & 0.4399 & 0.5264 & 0.4470 & 0.3576 \\
\hline
\end{tabular}}
\end{table}

\subsection{Inferential Statistics: Tropical Linear Discriminant Analysis}

We now present an example of an inferential task in statistics in palm tree space.  Classification on gene trees and phylogenetic trees is an important task in phylogenomics, such as comparative gene analysis and analysis on convergence on Bayesian inference on phylogenetic tree reconstruction \citep{Haws2021}.  Here, we propose tropical linear discriminant analysis (LDA), akin to classical LDA as a generalization of Fisher's linear discriminant \citep{Fisher}, to perform linear classification on gene trees.  LDA is a method in classical statistics which finds a linear combination of features to distinguish between two or more classes of objects or events, resulting in a linear classifier.  It is a supervised learning model that classifies categorical response variables from given explanatory variables and is constructed from a similar perspective as PCA described above in Section \ref{sec:treePCA}, which is to search for an explanatory linear subspace for variables.

\begin{definition}[Tropical Convex Hull]\label{def:polytope}
Take a finite subset $V = \{v^1, \ldots , v^m\}\subset \mathbb{R}^n/\mathbb{R}\one$.  The {\em tropical convex hull} or {\em tropical polytope} of $V$ is the smallest tropically-convex subset containing $V$, in the sense that the tropical line segment between any two points in $V$ is contained in $V$.   It can be written as the set of all tropical linear combinations of $V$ such that:
$$ \mathrm{tconv}(V) = \{a_1 \odot v^1 \oplus a_2 \odot v^2 \oplus \cdots \oplus a_m \odot v^m \mid  a_1,\ldots,a_m \in \mathbb{R} \}.$$
\end{definition}

Notice that a tropical polytope of a set of two points $\{v^1, v^2\} \subset \mathbb{R}^n/\mathbb{R}\one$ is the tropical line segment between $v^1, v^2$.

Let $\Gamma_{u, v}$ be the tropical line segment between $u, v \in \mathcal{U}_N$ and consider a tropical polytope
$$
\mathscr{P} = {\rm tconv}(w_1,w_2,\ldots,w_m) \subset \mathbb{R}^n/\mathbb{R}\one,
$$ 
where $w_1, \ldots , w_m$ are vertices of the tropical polytope $\mathscr{P}$.  Then, by formula 5.2.3 in \cite{maclagan2015introduction}, the projection map $\mathrm{proj}_{\mathscr{P}}$ sending any point $x \in \mathbb{R}^n/\mathbb{R}\one$ to a closest point in the tropical polytope $\mathscr{P} \subset \mathbb{R}^n/\mathbb{R}\one$ such that
$$ \mathrm{proj}_\mathscr{P} (x) \,= \,
\lambda_1 \odot  w^1 \,\oplus \,
\lambda_2 \odot  w^2 \,\oplus \, \cdots \,\oplus \,
\lambda_m \odot  w^m  ,
\quad {\text{where}} \quad \lambda_k = {\rm min}(x-w_k).
$$

Here, suppose we have a sample $\{(x_1, y_1),  \ldots , (x_m, y_m)\}$ where $x_1, \ldots , x_m \in \mathbb{R}^n/\mathbb{R}\one$ and $y_1, \ldots , y_m \in \{1, 2\}$.  Let $\mathcal{S}:= \{x_1, \ldots , x_m\} \subset \mathbb{R}^n/\mathbb{R}\one$.

\begin{definition}
A {\em tropical Fermat--Weber point} $F_{\mathcal{S}}$ of $\mathcal{S}$ is defined as
\[
FW_{\mathcal{S}} := \arg \min_{x \in \R^n/\mathbb{R}\one}\sum_{i=1}^m \dtr(x, x_i).
\]
Further, suppose we have $u, v \in \mathcal{U}_N$ and the tropical line segment between them $\Gamma_{u, v}$.  Then a {\em tropical Fermat--Weber point} $FW_{(\mathcal{S}, \Gamma_{u, v})}$ of $\mathcal{S}$ over $\Gamma_{u, v}$ is defined as
\[
FW_{(\mathcal{S}, \Gamma_{u, v})} := \arg \min_{x \in \Gamma_{u, v}}\sum_{i=1}^m \dtr(x, x'_i),
\]
where $x'_i$ is the projection of $x_i \in \mathcal{S}$ onto the tropical line segment $\Gamma_{u, v}$. 
\end{definition}

\begin{figure}
    \centering
\begin{tikzpicture}[scale=0.75]
\draw (0,0) -- (9,0);
\filldraw [black] (5,3) circle (3pt);
\node [above] at (5,3) {$FW'$}; 
\filldraw [black] (5,0) circle (3pt);
\node [below] at (5,0) {$FW$};
\filldraw [black] (0,0) circle (3pt);
\node [above] at (0,0) {$x_2$}; 
\filldraw [black] (9,0) circle (3pt);
\node [above] at (9,0) {$x_m$};
\filldraw [black] (7.5,0) circle (3pt);
\node [above] at (7.5,0) {$x_{m-1}$};
\filldraw [black] (5.5,0) circle (3pt);
\node [above] at (5.5,0) {$x_{k+1}$}; 
\node [above] at (6.5,0) {$\cdots$};
\filldraw [black] (4,0) circle (3pt);
\node [above] at (4,0) {$x_k$};
\filldraw [black] (1.5,0) circle (3pt);
\node [above] at (1.5,0) {$x_1$};
\filldraw [black] (2.25,0) circle (3pt);
\node [above] at (2.25,0) {$x_3$};
\node [above] at (3.125,0) {$\cdots$};
\end{tikzpicture}
    \caption{Sample points on the tropical line segment and their tropical Fermat--Weber points}
    \label{fig:lm1}
\end{figure} 

\begin{lemma}\label{lm1}
Suppose $u, v \in \mathbb{R}^n/\mathbb{R}\one$ and suppose $\mathcal{S}=\{x_1, \ldots , x_m\} \subset \Gamma_{u, v}$, where $\Gamma_{u, v}$ is the tropical line segment between $u, v \in \mathbb{R}^n/\mathbb{R}\one$.  Then a Fermat--Weber point $FW_{\mathcal{S}}$ of $\mathcal{S}$ is in $\Gamma_{u, v}$.
\end{lemma}
\begin{proof}
Let $FW \in \Gamma_{u, v}$ be the points such that
$
   \sum_{i=1}^{m} \dtr(FW, x_i) = \min_{x \in \Gamma_{u, v}} \sum_{i=1}^{m} \dtr(x, x_i)  
$
and suppose there exists $FW' \in \mathbb{R}^n/\mathbb{R}\one - \Gamma_{u, v}$ such that
\begin{equation}\label{eq:2}
    \sum_{i=1}^{m} \dtr(F', x_i)  < \sum_{i=1}^m \dtr(F, x_i).
\end{equation}
As seen in Figure \ref{fig:lm1}, by the fact that $\dtr$ is a metric over $\mathbb{R}^n/\mathbb{R}\one$, notice that $\dtr(FW, x_i) \leq \dtr(FW', x_{i}) +  \dtr(FW, FW')$ for all $i = 1, \ldots , m$.  Then we have
\[
\sum_{i=1}^{m} \dtr(FW, x_i) \leq \sum_{i=1}^{m} \Big(\dtr(FW', x_i) +  \dtr(FW, FW')\Big).
\]
Since $\dtr(FW, FW') \geq 0$, we have contradiction to \eqref{eq:2}.
\end{proof}

In our proposed tropical LDA, we follow suit of a classical linear discriminant analysis over a Euclidean space; here we assume we have only two labels (classes) in the response variable.  A classical LDA for predicting the binary response variable based on given explanatory variables finds a linear function which maximizes the distance between one centroid of projected points onto the linear function from observations in one group and another centroid of projected points from the other group.  Using Lemma \ref{lm1}, we adapt the classical LDA procedure to the tropical setting by learning a tropical line segment between two points over $\mathcal{U}_N$, $\Gamma_{u, v}$, which is a geodesic by Proposition \ref{prop:geodesic}.  In particular, we maximize the tropical distance between $FW_{(\mathcal{S}_1, \Gamma_{u, v})}$ and $F_{(\mathcal{S}_2, \Gamma_{u, v})}$, where $\mathcal{S}_1$ is a sample with label in the response variable $y_i = 1$ and  $\mathcal{S}_2$ is a sample with label in the response variable $y_i = 2$. More specifically, the problem is to find $u, v \in \mathcal{U}_n$ such that
\[
\max_{u, v, \in \mathcal{U}_n} \dtr(FW_{(\mathcal{S}_1, \Gamma_{u, v})}, FW_{(\mathcal{S}_2, \Gamma_{u, v})}).
\]

The algorithms for training and predicting a tropical LDA are given in Algoithms \ref{alg:LDA} and \ref{alg:LDA2}.

\begin{algorithm}
\caption{Training a Tropical Linear Discriminant Analysis (LDA)}\label{alg:LDA}
\begin{algorithmic}
\State {\bf Input}: $D_1 \subset \mathcal{U}_N$ is a set of ultrametrics with label $1$ and $D_2 \subset \mathcal{U}_N$ is a set of ultrametrics with label $2$.  Number of iterations $I \geq 1$.
\State {\bf Output}: End points $u, v \in \mathcal{U}_N$ for the tropical line segment for linear discriminant analysis.  
\State Pick a random point $x_1$ in $D_1$.
\State Pick a random point $x_2$ in $D_2$.
\State Compute a tropical Fermat--Weber point $FW_1\in \Gamma_{x_1, x_2}$ of projections of all points in $D_1$ over $\Gamma_{x_1, x_2}$ and compute a tropical Fermat--Weber point $FW_2 \in \Gamma_{x_1, x_2}$ of projections of all points in $D_2$  over $\Gamma_{x_1, x_2}$.
\State Let $y^*$ be the tropical distance between $FW_1$ and $FW_2$.
\State Set $u = x_1$ and $v = x_2$.
\For{$i = 1, \ldots , I$}
\State Set $y_0 = y^*$.
\State Set $u_0 = u$ and $v_0 = v$.
\State Pick random points $\hat{u}, \hat{v} \in \mathcal{U}_n$.
\State Compute a tropical Fermat--Weber point $FW_1 \in \Gamma_{\hat{u}, \hat{v}}$ of projections of all points in $D_1$ over $\Gamma_{\hat{u}, \hat{v}}$ and compute a tropical Fermat--Weber point $FW_2 \in \Gamma_{\hat{u}, \hat{v}}$ of projections of all points in $D_2$  over $\Gamma_{\hat{u}, \hat{v}}$.
\State Compute the tropical distance $y'$ between $FW_1$ and $FW_2$.
\If{$y' > y_0$}
    \State $y^* = y'$
    \State $u = \hat{u}$ and $v = \hat{v}$
\ElsIf{$y' \leq y_0$}
    \State $y^* = y_0$
\EndIf
\EndFor
\State {\bf Return}: $u, v \in \mathcal{U}_N$.
\end{algorithmic}
\end{algorithm}

\begin{algorithm}
\caption{Predicting a Tropical Linear Discriminant Analysis (LDA)}\label{alg:LDA2}
\begin{algorithmic}
\State {\bf Input}: Output $u, v \in \mathcal{U}_N$ from Algorithm \ref{alg:LDA}.  Test set $D_1=\{x_1^1, \ldots , x^1_{m_1}\} \subset \mathcal{U}_N$ with label 1 and Test set $D_2 =\{x_1^2, \ldots , x^2_{m_2}\} \subset \mathcal{U}_N$ with label 2. 
\State {\bf Output}:  Labels $L_1=\{L_1^1, \ldots , L_{m_1}^1\}$ and $L_2=\{L_1^1, \ldots , L_{m_1}^1\}$ for all points in $D_1$ and $D_2$, respectively.
\State $y = \frac{\dtr(u, v)}{2}$.
\For{each point $x_1^1, \ldots , x^1_{m_1} \in D_1$}
\State Compute $y_1 = \dtr(u, x_i^1)$ and $y_2 = \dtr(v, x_1^1)$.
\If{$y_1 < y$}
    \State Label $L_i^1 = 1$ for $x_i^1$.
\ElsIf{$y_2 \leq y$}
    \State Label $L_i^1 = 2$ for $x_i^1$.
\EndIf
\EndFor
\For{each point $x_2^1, \ldots , x^2_{m_1} \in D_2$}
\State Compute $y_1 = \dtr(u, x_i^2)$ and $y_2 = \dtr(v, x_1^2)$.
\If{$y_1 < y$}
    \State Label $L_i^2 = 1$ for $x_i^2$.
\ElsIf{$y_2 \leq y$}
    \State Label $L_i^2 = 2$ for $x_i^2$.
\EndIf
\EndFor
\State {\bf Return}: $L_1, \, L_2$.
\end{algorithmic}
\end{algorithm}

\subsubsection{Numerical Experiments: Mixture of Coalescent Models}

We simulated data following \cite{10.1093/bioinformatics/btaa564}; we generated gene trees with a species tree under a coalescent model via the software {\tt Mesquite} \citep{mesquite}.  We fixed
the effective population size $m = 100,000$ and varied $r := \frac{sd}{m}$, where $sd$ is the species depth.  We generated simulated data sets as described in Algorithm 5.1 of \cite{10.1093/bioinformatics/btaa564}.

The number of leaves was set to $N$ = 10. We used the ratio between species depth $sd$ and effective population size $r$ to vary $r = \{0.25, 0.5, 1, 2, 5, 10\}$.  We took an 80-20 split for training and test sets; the sample size of a training set was set to 120 and the sample size for a test set was set to 30.  The classification accuracy rates are shown in Table \ref{tab:res}.  In our simulations, we approximated the tropical LDA by sampling random points in $\mathcal{U}_N$ and approximated the projection of a tropical Fermat--Weber point by sampling random points for increased computational efficiency.  The number of iterations was set to $500$.  Larger $r$ results in more accurate classification results \citep{10.1093/bioinformatics/btaa564}.

\begin{table}
    \centering
    \begin{tabular}{c|cccccc}\hline
         $r$ & 0.25 & 0.5 & 1 & 2 & 5 & 10 \\\hline
         Accuracy Rate& 0.533& 0.717& 0.733& 0.967 &1.000 &1.000\\\hline
    \end{tabular}
    \caption{Classification accuracy rates for tropical LDA for simulated data sets generated under the coalescent model with two different species trees.}
    \label{tab:res}
\end{table}

\subsection*{Software and Data Availability}

Software to compute tropical LDA and both tropical and BHV PCA is publicly available in R and Java code.  The implementation of tree PCA to the pre-processed influenza data described in this paper is located on the FluPCA GitHub repository at \url{https://github.com/antheamonod/FluPCA}.  The resulting figures from both BHV and tropical PCA projections for all 21 data sets are also available on the FluPCA GitHub repository.


\section{Discussion}
\label{sec:end}

In this paper, we defined palm tree space as the space of phylogenetic trees with $N$ leaves, endowed with the tropical metric.  We gave results on its analytic, topological, geometric, and combinatorial properties and showed that they are conducive to both descriptive and inferential statistics, as well as unsupervised and supervised machine learning.  We have also shown that it is a setting amenable to rigorous treatments and studies in probability.  We performed both descriptive and inferential statistical analysis on real and simulated data and showed that the tropical setting is viable for such studies, and moreover, in our real data application, we see that the tropical approach outperforms the BHV approach.  

\subsection*{Biological and Statistical Implications}

Despite the statistical challenges of arbitrariness of dimension of BHV polytopes and stickiness which affect descriptive and inferential statistics, the BHV parameterization has been successfully implemented to reveal important biological findings (e.g., \cite{zairis2014moduli}).  In terms of interpretation, the unresolved singularities of BHV Fr\'{e}chet means translate to ``indecisiveness'' of which branching patterns or tree topologies are ``preferred,'' which is consistent with what is often seen in some biological settings where the trees arise from sequence alignment.  However, mathematically, trees are used to model other biological phenomena, such as pulmonary paths as airway trees \cite[e.g.,][]{6389677, 6985680}; brain growth and structure \cite[e.g.,][]{pmid24078809}; and neuronal morphologies \cite[e.g.,][]{Kanari2018}.  Such a probabilistic assumption may not be reasonable in these other settings.  Given recent research interest in developing methods to bypass these difficulties support the goal of our work, which is that exploring alternative representations is an important research direction \cite[e.g.,][]{anaya2020properties, doi:10.1137/15M1050914}.  An important and interesting direction for future research is the identification of non-uniform probability distributions in the tropical setting, which is challenging yet promising; various ways in which tropical Gaussian distributions may be constructed have been previously outlined \citep{tran2018tropical}.

In the context of shape statistics \citep{10.2307/2245331} and computational anatomy \citep{grenander1998computational}, the data objects of interest are often modeled as elements of algebraic spaces.  In particular, these algebraic spaces are quotient spaces generated by group actions.  Recent work has studied the behavior of estimators on such spaces, uncovering undesirable properties, such as biasedness and inconsistency when the group actions are random (that is, when the quotient spaces are generated by elements of the group are chosen at random to act on the topological space) or continuous (as in the case of Lie groups acting on Riemannian manifolds) \citep{doi:10.1137/16M1083931, 10.1007/978-3-319-25040-3_15}.  Nonparametric methods have been developed to bypass the problem of inconsistency \citep{BHATTACHARYA20141}.  As previously mentioned, the tropical projective torus $\mathbb{R}^n/\mathbb{R}\one$ is a quotient space that may be generated by a group action, however the biasedness and inconsistency in previous work arise due to the poor behavior of the transformed metric after it is mapped into the quotient space, which results in a pseudometric.  In our case, the tropical metric is well-behaved and defined directly on the quotient space, therefore differing in setting to previous work.

It should also be noted that the non-uniqueness property of geodesics in palm tree space poses computational difficulties, but does not prohibit statistical analysis and can still yield useful descriptive as well as inferential information, for example, on clustering behavior.  Another important setting where geodesics are not unique is that of positively-curved spaces.  The asymptotic behavior and distributions on Riemannian manifolds, including positively curved manifolds, has been studied in existing work \citep{10.2307/20535502}.  Recent work develops techniques for data analysis on curved spaces by tuning the geodesic metrics accordingly \citep{Kobayashi2019}.  In particular, a general Fr\'{e}chet function is defined, and its parameters are chosen accordingly, depending on the goal: for example, one geodesic metric may be transformed into another to control the curvature of the space for data analysis.  There are large bodies of existing work in related areas on curved spaces, for example, in the case of manifold learning; shape statistics; Wasserstein spaces for probability measures; and information geometry.  Though these domains each have their own specific goals and studies, data analysis and computation play a central role in these settings.  Moreover, there are known settings in these related areas where positive curvature, and thus non-uniqueness of geodesics, arises (for example, the 2-Wasserstein space for Gaussian measures is positively curved \citep{takatsu2011}).  Adapting existing techniques in these settings to statistical analysis in palm tree space is an important direction of research that merits exploration.

\subsection*{Future Work}

Our work invites the reinterpretation of existing statistical methods in terms of the tropical metric to make a wider array of exact analyses readily available and interpretable to phylogenetic research with potential impact for biological discoveries.  Two important directions for future work include the search for explicit parametric probability distributions in order to reinterpret classical probability-based statistical approaches on sets of phylogenetic trees, based on, for example, the classical central limit theorem.  Another is the formal definition and computation of tropical Fréchet means and a study of the extent of its stickiness.  Finally, a complete and systematic comparative study between palm tree and BHV space would be an interesting study to conduct.


\section*{Acknowledgments}

The authors are grateful to Robert J.~Adler, Omer Bobrowski, Yueqi Cao, Ioan Filip, Maia Fraser, Stephan Huckemann, Elliot Paquette, Juan \'{A}ngel Pati\~{n}o-Galindo, Yitzchak (Elchanan) Solomon, and Bernd Sturmfels for helpful discussions.  We are especially indebted to Carlos Am\'endola and Hossein Khiabanian for their extensive commentary, suggestions, and support throughout the course of this project.  

A.M.~and R.Y.~wish to acknowledge the Mathematisches Forschunginstitut Oberwolfach (MFO) for hosting their visit in January 2018, which inspired this work.  B.L.~wishes to acknowledge support by the Simons Foundation and the hospitality of the Institut Mittag-Leffler during the spring of 2018, where a large part of this research was carried out.

The authors would also like to acknowledge the GI-SAID EpiFlu\textsuperscript{TM} initiative for making the data used in this study publicly available.

R.Y.~is supported in part by NSF DMS \#1622369 and \#1916037.  Any opinions, findings, and conclusions or recommendations expressed in this material are those of the author(s) and do not necessarily reflect the views of any of the funders.

\subsection*{Author Contributions Statement}

A.M.~conceived the study and designed the research.  A.M., B.L., and R.Y.~performed research.  A.M.~and R.Y.~developed the algorithms; Q.K.~implemented the software and performed the analyses.  A.M.~wrote the manuscript.

\subsection*{Competing Financial Interests Statement}

The authors declare that no competing financial interests exist.








\clearpage
\newpage
\bibliographystyle{chicago}  
\bibliography{PalmTree_Ref} 


\end{document}